\documentclass{amsart}
\usepackage{amssymb}

\usepackage[utf8]{inputenc}
\usepackage[T1]{fontenc}
\usepackage{mathpazo,courier}
\usepackage[scaled]{helvet}

\usepackage[usenames,dvipsnames]{xcolor}
\definecolor{darkblue}{rgb}{0.0, 0.0, 0.55}
\usepackage[pagebackref,colorlinks,linkcolor=BrickRed,citecolor=OliveGreen,urlcolor=darkblue,hypertexnames=true]{hyperref}

\usepackage{enumerate}

\newcommand\N{\mathbb N}
\newcommand\Z{\mathbb Z}
\newcommand\Q{\mathbb Q}
\newcommand\R{\mathbb R}

\usepackage{ushort}
\newcommand\x{\ushort X}

\newcommand\g{\ushort g}

\newcommand\al\alpha
\newcommand\la\lambda
\newcommand\de\delta
\newcommand\ep\varepsilon
\newcommand\ph\varphi
\newcommand\si\sigma
\newcommand\ta\tau
\newcommand\ps\psi
\newcommand\Ph\Phi
\newcommand\Ps\Psi

\renewcommand\O{\mathcal O}
\newcommand\m{\mathfrak m}

\newcommand\bigdotcup{\mathop{\dot\bigcup}}

\DeclareMathOperator\hess{Hess}
\DeclareMathOperator\conv{conv}

\DeclareMathOperator\sgn{sgn}
\DeclareMathOperator\st{st}

\DeclareMathOperator\tr{tr}
\DeclareMathOperator\id{id}
\DeclareMathOperator\extr{extr}
\DeclareMathOperator\convbd{convbd}
\DeclareMathOperator\lasserre{Lasserre}

\theoremstyle{definition}
\newtheorem{thm}{Theorem}
\newtheorem{cor}[thm]{Corollary}
\newtheorem{pro}[thm]{Proposition}
\newtheorem{rem}[thm]{Remark}
\newtheorem{ex}[thm]{Example}

\newtheorem{lem}[thm]{Lemma}
\newtheorem{df}[thm]{Definition}
\newtheorem{nt}[thm]{Notation}

\usepackage{stmaryrd} % for \mapsfrom and \lightning

\title[Lasserre relaxations and pure states]
{On the exactness of Lasserre relaxations and\\pure states over real closed fields}

\author[T.-L. Kriel]{Tom-Lukas Kriel}
\address{Fachbereich Mathematik und Statistik, Universität Konstanz, 78457 Konstanz, Germany}
\email{tom-lukas.kriel@uni-konstanz.de}

\author[M. Schweighofer]{Markus Schweighofer}
\address{Fachbereich Mathematik und Statistik, Universität Konstanz, 78457 Konstanz, Germany}
\email{markus.schweighofer@uni-konstanz.de}


\date{September 13, 2018}


\begin{document}
\begin{abstract}
Consider a finite system of non-strict polynomial inequalities with solution set $S\subseteq\R^n$.
Its Lasserre relaxation of degree $d$ is a certain natural linear matrix inequality in the original variables
and one additional variable for each nonlinear monomial of degree at most $d$. It defines a spectrahedron
that projects down to a convex semialgebraic set containing $S$. In the best case, the projection equals the convex hull of $S$.
We show that this is very often the case for sufficiently high $d$
if $S$ is compact and ``bulges outwards'' on the boundary of its convex hull.

Now let additionally a polynomial objective function $f$ be given, i.e., consider a polynomial optimization problem.
Its Lasserre relaxation of degree $d$ is now a semidefinite program.
In the best case, the optimal values of the polynomial optimization problem and its relaxation agree.
We prove that this often happens if $S$ is compact and $d$ exceeds some bound that depends on the description of
$S$ and certain characteristicae of $f$ like the mutual distance of its global minimizers on $S$.
\end{abstract}

\maketitle
\newpage
\tableofcontents
\newpage

%********************************************************************************
% Section: Introduction
%********************************************************************************

\section{Introduction}

\subsection{Basic notation}
Throughout the article, $\N$ and $\N_0$ denote the set of positive and nonnegative integers,
respectively. All rings will be commutative with $1$ and, correspondingly, ring homomorphisms always map $1$
to $1$ by definition. We fix $n\in\N_0$ and denote by $\x:=(X_1,\dots,X_n)$ a tuple of $n$ variables. For any ring
$A$ in which $0\ne1$, we write \[A[\x]:=A[X_1,\dots,X_n]\] for the polynomial ring in these variables over $A$.
For $\al\in\N_0^n$, we denote
\[|\al|:=\al_1+\ldots+\al_n\qquad\text{and}\qquad \x^\al:=X_1^{\al_1}\dotsm X_n^{\al_n}.\]
If $A$ is a ring and $p=\sum_\al a_\al\x^\al\in A[\x]$
with all $a_\al\in A$, the degree of $p$ is defined as
\[\deg p:=\max\{|\al|\mid a_\al\ne0\}\]
if $p\ne0$ and $\deg p:=-\infty$ if $p=0$. 
For each ring $A$ and all $d\in\N_0$, we set
\[A[\x]_d:=\{p\in A[\x]\mid\deg p\le d\}.\]
In particular, $\R[\x]_d$ is the real vector space of all real polynomials of degree at most $d$.
We call polynomials of degree at most $1$ \emph{linear} polynomials, of degree at most $2$ \emph{quadratic} polynomials
and so on. We call a polynomial a \emph{linear form} if \emph{all of its monomials} are linear.

\subsection{Systems of polynomial inequalities and polynomial optimization problems}
Fix $\g=(g_1,\ldots,g_m)\in\R[\x]^m$, consider the \emph{basic closed semialgebraic} set \cite[Chapter 2]{pd1}
\[S(\g):=\{x\in\R^n\mid g_1(x)\ge0,\ldots,g_m(x)\ge0\}.\]
Given $\g$, one would like to understand $S(\g)$ as well as possible. This is the task of
\emph{solving systems of polynomial inequalities}. Moreover, given in addition $f\in\R[\x]$,
one would like to compute the infimum (or the supremum) of $f$ on $S(\g)$ and extract, if the infimum is a minimum,
the corresponding minimizers. This is the task of \emph{polynomial optimization}.
Both mentioned tasks are very difficult problems in general. Lasserre and others developed a by now well-established
and successful symbolic-numeric technique to approach these problems \cite{l1,l2,lau,mar}.

\subsection{The idea behind Lasserre relaxations}
The idea is to fix a so-called \emph{relaxation degree $d$} and to generate from given polynomial inequalities new families
of polynomial inequalities of degree at most $d$ by multiplying the $g_i$ with squares of polynomials of appropriate degree.
These families of \emph{polynomial} inequalities are then turned into families of \emph{linear} inequalities
by substituting the non-linear monomials of degree at most $d$ with new variables. The gist of the method is that this family of
linear inequalities corresponds to a single \emph{linear matrix inequality} which is called the
\emph{degree $d$ Lasserre relaxation} of the given system of polynomial inequalities.
In the case of polynomial optimization, the objective polynomial $f$ is also linearized in the same way and together with the
linear matrix inequality now forms a \emph{semidefinite program} which is called the \emph{degree $d$ Lasserre relaxation} of
the given polynomial optimization problem. For the purpose of this article, we will take a shortcut and will not formally
define the Lasserre relaxations themselves but instead directly define in a rather abstract way
the relevant convex supersets $S_d(\g)$ of $S(\g)$
(the projection of the \emph{spectrahedron} defined by the linear matrix inequality) and the optimal value
$\lasserre_d(f,\g)$ (the optimal value of the semidefinite program). For the study of this article, it is hence not even
necessary to know what is a linear matrix inequality or a semidefinite program. We refer however the
readers that are not familiar with the Lasserre relaxation but want to understand why it is such a natural procedure to \cite{ks,lau,mar}. Occasionally, the Lasserre relaxation is also called \emph{moment relaxation} but we prefer to name it after Jean Lasserre
who played a leading role in the invention and investigation of this natural procedure \cite{l1,l2}.

\subsection{Formal definitions related to Lasserre relaxations}\label{subs:formal}
We now fix a relaxation degree $d\in\N_0$. In order to define $S_d(\g)$ and $\lasserre_d(f,\g)$ as promised, we first introduce the
\emph{$d$-truncated quadratic module} $M_d(\g)$ associated to $\g$ by
\[
M_d(\g):=
\left\{\sum_{i=0}^m\sum_jp_{ij}^2g_i\mid p_{ij}\in\R[\x],\deg\left(p_{ij}^2g_i\right)\le d\right\}\subseteq M(\g)\cap\R[\x]_d
\]
where $g_0:=1\in\R[\x]$. Here the inner sum over $j$ is a finite sum of arbitrary length. Moreover, it will be
of utmost importance that the degree of each individual $p_{ij}^2g_i$ is restricted by $d$
(and not just the degree of the outer sum over $i$ which would make the inclusion into an equality).
Note that $M_d(\g)$ consists of polynomials of degree at most $d$ that are very obviously nonnegative on the whole of $S(\g)$.
Now we set
\begin{gather*}
S_d(\g):=\left\{(L(X_1),\ldots,L(X_n))~\middle|~
\begin{split}
L\colon\R[\x]_d\to\R\text{ linear},\\
L(M_d(\g))\subseteq\R_{\ge0},\\
L(1)=1
\end{split}
\right\}\subseteq\R^n
\end{gather*}
and for $f\in\R[\x]_d$
\begin{gather*}
\lasserre_d(f,\g):=\inf
\left\{L(f)~\middle|~
\begin{split}
L\colon\R[\x]_d\to\R\text{ linear},\\
L(M_d(\g))\subseteq\R_{\ge0},\\
L(1)=1
\end{split}
\right\}\in\{-\infty\}\cup\R\cup\{\infty\}
\end{gather*}
where the infimum is taken in the ordered set $\{-\infty\}\cup\R\cup\{\infty\}$. Considering the evaluations of degree at most $d$
polynomials at points of $S(\g)$, it is clear that $S_d(\g)$ contains $S(\g)$ and
$\lasserre_d(f,\g)$ is a lower bound of the infimum of $f$ on $S(\g)$ for all $d\in\N_0$.

\subsection{Our contribution towards solving systems of polynomial inequalities}\label{subs:towardssolving}
We recall that a subset $S$ of a real vector space is called \emph{convex} if
\[\la x+(1-\la)y\in S\]
for all $x,y\in S$ and $\la\in[0,1]$.
There is always a smallest convex subset containing $S$. It is called the \emph{convex hull} of $S$ and we denote it by $\conv S$.
Since each $S_d(\g)$ is obviously a convex set, it is easy to see that
\[S(\g)\subseteq\conv S(\g)\subseteq\quad\ldots\quad\subseteq S_3(\g)\subseteq S_2(\g)\subseteq S_1(\g)\subseteq S_0(\g).\]
The best one can hope for is thus that \[\conv S(\g)=S_d(\g)\]
for some $d$ (unfortunately, we do not prove anything about the dependance of $d$ on $\g$).
In Corollary \ref{lasserreexact} below, we prove that this best possible case occurs (i.e., the Lasserre relaxation eventually becomes
exact) if $S(\g)$ is compact
and mild additional hypotheses are met:
\begin{itemize}
\item In fact, $M(\g):=\bigcup_{d \in \N} M_d(\g)$ has to be \emph{Archimedean}, see Definition \ref{defqm}, Remark \ref{genmodule} and Proposition~\ref{archmodulechar} below.
This implies compactness of $S(\g)$ but is in general not equivalent to it. However, up to
changing the description $\g$ of $S(\g)$ which is easy
in most practical cases, it \emph{is} equivalent to it as explained in Remark \ref{notmuchstronger} below.
\item We assume $S(\g)$ to have nonempty interior near its convex boundary in the sense of Definition \ref{nearconvexboundary}.
This is for example satisfied if the semialgebraic set $S(\g)$ is locally full-dimensional \cite[Definition 2.8.11]{bcr} at each of its points.
\item Roughly speaking, at each point of $S(\g)$ which lies on the boundary of its convex hull, the smooth algebraic
hypersurfaces that 
``confine'' $S(\g)$ must have positive curvature (where the normal is chosen in a way that respects the convexity of $S(\g)$),
in particular $S(\g)$ satisfies a kind of second order strict convexity condition locally on the boundary of its convex hull.
See Definitions~\ref{nearconvexboundary} and \ref{dfconcave} as well as Remarks \ref{whatmeansquasiconcavity}
and \ref{mildnatural} for the exact condition and more details.
With the right description $\g$ of $S(\g)$, this often just means that $S(\g)$ has
``no flat parts'' on the boundary of its convex hull.
\end{itemize}
This last ``no flat borders'' condition is the most severe since it can easily be violated if one of the $g_i$ is linear (i.e., $\deg g_i\le1$)
which happens quite often in the applications. As shown in \cite[Example 4.10]{ks}, it can however not be omitted. The authors have
however proven a variant of Corollary \ref{lasserreexact} that allows such linear inequalities and
more generally $\g$-sos-concave inequalities \cite[Main Theorem 4.8]{ks} in the case where $S(\g)$ is convex. 
For convex $S(\g)$, our Corollary~\ref{lasserreexact} is thus weaker than \cite[Main Theorem 4.8]{ks}.
Whereas \cite{ks} builds up on the seminal work of Helton and Nie \cite{hn1,hn2} based on representations of
positive matrix polynomials, we use in this work completely different techniques, namely the technique of pure states we will discuss below. Note that we do not know how to prove \cite[Main Theorem 4.8]{ks} using pure states, and conversely we did not succeed
to prove Corollary \ref{lasserreexact} in this article using positivity certificates for matrix polynomials. 
The results and techniques of this article and of \cite{ks} thus nicely complement each other. While we pushed in \cite{ks}
the technique of Helton and Nie \cite{hn1,hn2} to its limits, we develop in this article a completely new approach through pure states.

\subsection{Our contribution to polynomial optimization}\label{subs:our}
Next let us explain, what we can prove about solving polynomial optimization problems. To this end, fix some objective
$f\in\R[\x]_d$. Again, it is easy to see that
\[\inf\{f(x)\mid x\in S(\g)\}\ge\quad\ldots\quad\ge\lasserre_{d+1}(f,\g)\ge\lasserre_d(f,\g).\]
The best one can hope for is thus that \[\inf\{f(x)\mid x\in S(\g)\}=\lasserre_d(f,\g)\]
for some $d$. It follows easily from Scheiderer's theorem \cite[Corollary 3.6]{s1} (see also \cite[Theorem 9.5.3]{mar})
which we reprove in Corollary \ref{putinarzeros} below
that this best possible case occurs (i.e., the Lasserre relaxation eventually becomes
exact) if $S(\g)$ is compact and nonempty, $f$ has only finitely many global minimizers on $S(\g)$ which lie all in the
interior of $S(\g)$ and if the following mild additional hypotheses are satisfied:
\begin{itemize}
\item We suppose again that $M(\g)$ is \emph{Archimedean}.
\item We suppose that the second order sufficient condition for a local minimum is satisfied at each of the
global minimizers of $f$ on $S(\g)$, i.e., the Hessian of $f$ is positive definite at these points. (Note that
the second order necessary condition for a local minimum already guarantees that the Hessian is
positive \emph{semi}definite at these minimizers.)
\end{itemize}
Our contribution is now that Theorem \ref{optimizationbound} is to the best of our knowledge the first known
result about at which degree the Lasserre relaxation of a polynomial
optimization problem becomes exact, namely we prove that the corresponding $d$ depends on
\begin{itemize}
\item the description $\g$ of $S(\g)$,
\item an upper bound on the degree of $f$,
\item an upper bound on the absolute value of the coefficients of $f$,
\item the number $k$ of the finitely many distinct minimizers $x_1,\dots,x_k$ of $f$ on $S(\g)$
(all lying in the interior of $S(\g)$),
\item a lower bound on the mutual distance of these minimizers and on their distance to $\R^n\setminus S(\g)$ and
\item a positive lower bound of
\[S\setminus\{x_1,\dots,x_k\}\to\R_{>0},\ x\mapsto\frac{f(x)-\min\{f(\xi)\mid\xi\in S(\g)\}}{u(x)}\]
where $u$ is a certain canonical nonnegative polynomial with real zeros $x_1,\dots,x_n$
(this bound will be very small if $f$ takes values close to \[\min\{f(\xi)\mid\xi\in S(\g)\}\] at
\emph{local nonglobal} minimizers on $S(\g)$ or if $f$ does not grow quickly enough
around its global minimizers on $S(\g)$).
\end{itemize}
In other words, for fixed $\g$ and for a fixed threshold for the algebraic and geometric complexity of an objective polynomial,
we prove that there exists a degree $d$ in which the Lasserre relaxation of the polynomial optimization problem becomes exact for
all objective functions $f$ whose complexity does not succeed this threshold. Interestingly, our result seems to suggest that
Lasserre's method for polynomial optimization might have difficulties with some of the phenomena that
classical optimization algorithms also struggle with like for example local minimizers at which the objective function
has a value close but different from the global minimum.

\subsection{Qualitative versus quantitative bounds}
Neither in Corollary \ref{lasserreexact} concerning systems of polynomial inequalities
nor in Theorem \ref{optimizationbound} concerning polynomial optimization problems,
we prove anything concrete on the way, the relaxation degree $d$ of exactness depends on $\g$ and the complexity of $f$.
This is because in the proof we work with infinitesimal numbers rather than
with real quantities so that our analysis gives the qualitative rather than the quantitative behavior.
It is conceivable to find a quantitative version of the proofs and more concrete bounds by working with real
numbers only instead of real closed fields. At least for general $\g$, this would however be an enormous effort.

\subsection{Pure states}
The concept of pure states stems from functional analysis and has recently entered commutative algebra \cite{bss}.
Pure states on rings often just correspond to ring homomorphisms into the real numbers
and therefore can be seen as real points of a variety \cite[Propositions 4.4. and 4.16]{bss}.
In this article, the variety concerned is just affine space.
Although, this has not been investigated thoroughly and it seems speculative, we feel that pure states on ideals of a ring
often just correspond to a certain kind up of ``blowup'' at the subvariety defined by the ideal, i.e., the subvariety is replaced
by the directions pointing out from it. All this should be understood rather in a scheme theoretic manner, i.e., a kind of
``multiplicity information'' is present. In our case, we will remove the finite subvariety of global minimizers
and replace it by second order directional derivatives pointing out from it. These ideas are not new
\cite[Theorem 7.11]{bss} but here we add another important idea: We extend the technique of pure states (still real valued!)
to work over \emph{real closed extension fields of $\R$} (see Subsection \ref{subs:rcf} below) rather than over $\R$ itself.
More precisely, we work over the subring
of finite elements of a real closed extension field of $\R$. This leads to new phenomena which we have to deal with.
It will turn out that the standard part map from this ring to the reals will play a very important role (see Theorem \ref{dicho} below).

\subsection{Positive polynomials}\label{subs:pospol}
Most of the rest of the article is written from the perspective that is dual to Lasserre relaxations, namely from the perspective of sums of
squares representations of positive polynomials. Viewed from this angle, our first main result is
Theorem \ref{putinarzerosdegreebound} from which Theorem \ref{optimizationbound} follows. It is a generalization
of Putinar's Positivstellensatz \cite[Lemma 4.1]{put}
where the polynomial to be represented is allowed to have (finitely many) zeros (of a certain kind)
and which provides at the same time qualitative (but not quantitative) degree bounds.
This result generalizes both Putinar's Positivstellensatz with degree bounds (a qualitative version of which follows already
from the ideas in \cite{pre} and \cite[Chapter 8]{pd1}, see \cite{ns} for a quantitative version) and Scheiderer's generalization of Putinar's Positivstellensatz to polynomials with zeros \cite[Corollary 3.6]{s1}. Due to our new proof technique,
these latter two results appear in our work as Corollaries \ref{putinardegreebound}  and \ref{putinarzeros}.
Our second main result on positive polynomials is Corollary \ref{linearstabilitybound}
from which Corollary \ref{lasserreexact} follows. The proofs of this representation theorem for linear polynomials and
of the preparatory Lemmata \ref{prep1}, \ref{lagrange2} and \ref{finitecontact} are very subtle.

%********************************************************************************
% Section: Prerequisites
%********************************************************************************

\section{Prerequisites}

In order to make this article more accessible for readers from the diverse backgrounds such as
optimization, numerical analysis, real algebraic geometry and convexity, we collect in this section the necessary
tools and methods and present them in way that is tailored towards our needs.

\subsection{Notation}\label{subs:notation}

For a subset $S$ of $\R^n$, we write $\overline S$ and $S^\circ$ for its closure and its interior, respectively.
Moreover, $\partial S=\overline S\setminus S^\circ$ denotes its boundary.
For $\ep\in\R_{\ge0}$ and $x\in\R^n$, we denote by
\[B_\ep(x):=\{y\in\R^n\mid\|x-y\|<\ep\}\]
the \emph{open ball of radius $\ep$ centered at $x$} (for $\ep=0$ it is empty). Consequently,
\[\overline{B_\ep(x)}:=\{y\in\R^n\mid\|x-y\|\le\ep\}\]
is the \emph{closed ball of radius $\ep$ centered at $x$} for $\ep\in\R_{>0}$.
 If $f$ is a polynomial that can be evaluated on
a set $S$, we simply say
\[f\ge0\text{ on }S\]
to express that $f(x)\ge0$ for all $x\in S$. If $A$ and $B$ are subsets of an additively written abelian group (e.g., the additive
group of a ring or a vector space), then we use suggestive notations like
\begin{align*}
A+B&:=\{a+b\mid a\in A,b\in B\}\qquad\text{and}\\
\sum A&:=\{a_1+\ldots+a_k\mid k\in\N_0,a_1,\ldots,a_k\in A\}.
\end{align*}
We use similar suggestive notations with respect to the multiplication in a ring or scalar multiplication in a vector space.
If $A$ is a ring, then
$A^2$ will often stand for the set
\[A^2:=\{a^2\mid a\in A\}\]
of squares in $A$ and consequently
\[\sum A^2=\{a_1^2+\ldots+a_k^2\mid k\in\N_0,a_1,\ldots,a_k\in A\}\]
for the set of sums of squares in $A$.
This could conflict with the common notation
\[A^m:=\underbrace{A\times\ldots\times A}_{m\text{ times}}\]
for the $m$-fold Cartesian product but from the context it should always be clear what we mean.
Another source of confusion could be that we also use the notation $IJ$ from commutative algebra
for ideals $I$ and $J$ of a ring $A$ to denote their product, i.e.,
\[IJ:=\sum IJ\]
where the right hand side is written in our suggestive notation.
Correspondingly, we write $I^2$ for the product of the ideal $I$ with itself, i.e.,
\[I^2:=II:=\sum II\]
and similarly for higher powers of the ideal $I$. If $N\in\Z$, we often write $N$ and actually mean its image under the unique
ring homomorphism $\Z\to A$.

\subsection{Quadratic modules}

\begin{df}\label{defqm}
Let $A$ be a ring and $M\subseteq A$. Then $M$ is called a \emph{quadratic module} of $A$
if $\{0,1\}\subseteq M$, $M+M\subseteq M$ and $A^2M\subseteq M$.
\end{df}

\begin{df}\label{bamu}
Let $A$ be a ring and $M$ a quadratic module of $A$. For $u\in A$, we set
\[B_{(A,M,u)}:=\{a\in A\mid\exists N\in\N:Nu\pm a\in M\}\]
and call its elements with respect to $M$ by $u$ \emph{arithmetically bounded} elements of $A$.
If $u=1$, then we write $B_{(A,M,u)}:=B_{(A,M)}$ and omit the specification ``by $u$''.
The quadratic module $M$ is called \emph{Archimedean} if $B_{(A,M)}=A$.
\end{df}

\begin{pro}\label{bammodule}
Suppose $\frac12\in A$ and $M$ is a quadratic module of $A$. Then $B_{(A,M)}$ is a subring of $A$ such that
\[a^2\in B_{(A,M)}\implies a\in B_{(A,M)}\]
for all $a\in A$ and $B_{(A,M,u)}$ is a
$B_{(A,M)}$-submodule of $A$ for each $u\in\sum A^2$.
\end{pro}

\begin{proof}
If $N\in\N$ with $(N-1)-a^2\in M$, then
\[N\pm a=(N-1)-a^2+\left(\frac12\pm a\right)^2+3\left(\frac12\right)^2\in M.\]
The rest follows from \cite[Proposition 3.2(d)]{bss}.
\end{proof}

\begin{rem}\label{genmodule}
If $A$ is a ring and $G\subseteq A$, then \[\sum_{g\in G\cup\{1\}}\sum A^2g\] is the smallest quadratic module
of $A$ containing $G$. It is called the quadratic module \emph{generated} by $G$ (in $A$). If $\g=(g_1,\dots,g_m)\in A^m$ is a tuple,
then the quadratic module generated by $\{g_1,\ldots,g_m\}$ equals
\[\sum_{i=0}^m\sum A^2g_i\]
where we set $g_0:=1\in A$ and
we simply call it the quadratic module generated by $\g$. For $\g\in\R[\x]^m$, we 
introduce the notation $M(\g)$ for the quadratic module generated by $\g$ in $\R[\x]$ so that we obviously have
\[M(\g)=\bigcup_{d\in\N_0}M_d(\g)\]
where $M_d(\g)$ is the $d$-truncated quadratic module associated to $\g$ introduced in Subsection \ref{subs:formal} above.
\end{rem}

\begin{pro}\label{archmodulechar}
Let $M$ be a quadratic module of $\R[\x]$. Then the following are equivalent:
\begin{enumerate}[(a)]
\item $M$ is Archimedean.
\item There is some $N\in\N$ such that $N-(X_1^2+\ldots+X_n^2)\in M$.
\item There are $m\in\N$ and $\g\in(\R[\x]_1\cap M)^m$
such that the polyhedron $S(\g)$ is nonempty and compact.
\item For each $f\in\R[\x]_1$, there is some $N\in\N$ such that $N+f\in M$.
\end{enumerate}
\end{pro}

\begin{proof}
This is well known (see for example \cite[Lemma 5.1.13]{pd1} and \cite[Cor. 5.2.4]{mar}).
A short proof can be found in \cite[Proposition 2.7]{ks}.
\end{proof}

\begin{rem}\label{prestel}
There are deep criteria on when a quadratic module is Archimedean which are important for applications.
We will not need them here.
Therefore we mention only one of them \cite[Theorem 6.3.6, Corollary 6.3.7]{pd1} and refer to \cite[Chapter 6]{pd1} for the general case:
For any quadratic module $M$ of $\R[\x]$, the following are equivalent:
\begin{enumerate}[(a)]
\item There are $g,h\in M$ with compact $S(g,h)$.
\item $M$ is Archimedean.
\end{enumerate}
\end{rem}

\begin{rem}\label{notmuchstronger}
For $n\ge2$, there are examples of $\g\in\R[\x]^m$
with compact (even empty) $S(\g)$ such that $M(\g)$ is not Archimedean
(see \cite[Example~7.3.1]{mar} or \cite[Example~6.3.1]{pd1}). However,
if one knows a big ball containing $S(\g)$, it suffices to add its defining
quadratic polynomial to $\g$ by Proposition \ref{archmodulechar}(b).
That is why for many practical purposes, the Archimedean
property of $M(\g)$ is not much stronger than the compactness of $S(\g)$.
\end{rem}

\subsection{Convex sets and cones}
In Subsection \ref{subs:towardssolving}, we have already reminded the reader about convex sets and convex hulls.
If $S$ is a convex set, then a point $x\in S$ is called an \emph{extreme point} of $S$ if there are no
$y,z\in S$ such that $y\ne z$ and $x=\frac{y+z}2$. We call $S$ a \emph{(convex) cone}
if $0\in S$, $S+S\subseteq S$ and $\R_{\ge0}S\subseteq S$.
We call sets of the form \[\{x\in\R^n\mid f(x)\ge0\}\]
where $f\in\R[\x]_1$ \emph{affine half-spaces}.
By \cite[Theorem 11.5]{roc}, the closure of the convex hull of $S$ is the intersection over all half-spaces containing $S$
(where the empty intersection is understood to be $\R^n$), i.e.,
\[\overline{\conv S}=\{x\in\R^n\mid\forall f\in\R[\x]_1:(\text{$f\ge0$ on $S$}\implies f(x)\ge0)\}.\]
Here we can omit the closure if $S$ is compact since
the closure of a compact subset of $\R^n$ is again compact by \cite[Theorem 17.2]{roc}.

\subsection{Pure states}

\begin{df}\label{defunit}
Let $C$ be a cone in the real vector space $V$ and $u\in V$.
Then $u$ is called a \emph{unit} for $C$ (in $V$) if for every $x\in V$ there is some
$N\in\N$ with $Nu+x\in C$.
\end{df}

\begin{rem}\label{linhullrem}
If $u$ is a unit for the cone $C$ in the real vector space $V$, then $u\in C$ and $C-C=V$.
\end{rem}

\begin{df}\label{defstate}
Let $V$ be a real vector space, $C\subseteq V$ and $u\in V$. A \emph{state}
of $(V,C,u)$ is a linear function $\ph\colon V\to\R$ satisfying
$\ph(C)\subseteq\R_{\ge0}$ and $\ph(u)=1$. We refer to the set
$S(V,C,u)\subseteq\R^V$ of all states of $(V,C,u)$ as the \emph{state space}
of $(V,C,u)$. It is a convex subset of the vector space $\R^V$ of all real-valued functions on $V$.
We call an extreme point of this convex set a \emph{pure state} of $(V,C,u)$.
\end{df}

\begin{rem}
In the previous definition, a pure state is thus an extreme point of an affine slice of the dual cone of $C$. The slice is given by intersecting with the affine hyperplane
consisting of linear forms on $V$ evaluating to $1$ at $u$. The idea will be to investigate $C$ by first investigating its state space.
\end{rem}

\begin{lem}\label{autolin}
Let $C$ be a cone in the real vector space $V$ satisfying $V=C-C$ and suppose
$\ph\colon V\to\R$ is a group homomorphism
from the additive group of $V$ into the additive group of the reals with $\ph(C)\subseteq\R_{\ge0}$. Then
$\ph$ is linear, i.e., $\ph(\la x)=\la x$ for all $\la\in\R$ and $x\in V$.
\end{lem}

\begin{proof}
Let $x\in V$. We have to show $\ph(\la x)=\la x$ for all $\la\in\R$. We know it for all $\la\in\Z$ and thus even for all $\la\in\Q$.
Because of Remark \ref{linhullrem}, we can suppose WLOG that $x\in C$. From $\ph(\R_{\ge0}x)\subseteq\ph(C)\subseteq\R_{\ge0}$,
we get easily the two inequalities in
\[r\ph(x)=\ph(rx)\le\ph(\la x)\le\ph(sx)=s\ph(x)\]
for all $r,s\in\Q$ with $r\le\la\le s$. Letting $r$ and $s$ tend to
$\la$, it follows that $\la\ph(x)\le\ph(\la x)\le\la\ph(x)$ as desired.
\end{proof}

\begin{thm}\label{conemembershipunitextreme}
Suppose $u$ is a unit for the cone $C$ in the real vector space $V$ and
let $x\in V$. Then the following are equivalent:
\begin{enumerate}[\rm(a)]
\item $\forall\ph\in\extr S(V,C,u):\ph(x)>0$
\item $\forall\ph\in S(V,C,u):\ph(x)>0$
\item $\exists N\in\N:x\in\frac1Nu+C$
\item $x$ is a unit for $C$.
\end{enumerate}
\end{thm}

\begin{proof}
Due to Lemma \ref{autolin}, this is \cite[Corollary 6.4]{goo}. See also \cite[Corollary 2.7]{bss}.
Alternatively, the proof is an exercise combining the ``basic separation theorem'' 
\cite[Page 15, §4B, Corollary]{hol}, the Banach-Alaoglu theorem \cite[Page 70, §12D, Corollary 1]{hol}
and the Krein-Milman theorem \cite[Page 74, §13B, Theorem]{hol}.
\end{proof}

\begin{cor}\label{conemembershipextr}
Suppose $u$ is a unit for the cone $C$ in the real vector space $V$ and
let $x\in V$. If $\ph(x)>0$ for all pure states $\ph$ of $(V,C,u)$, then
$x\in C$.
\end{cor}

\begin{thm}[Burgdorf, Scheiderer, Schweighofer]\label{dichotomy}
Let $A$ be a ring containing $\R$.
Suppose that $I$ is an ideal and $M$ a quadratic module of $A$,
$u$ is a unit for $M\cap I$ in
$I$ and $\ph$ is a pure state of $(I,I\cap M,u)$.
Then
\[\Ph\colon A\to\R,\ a\mapsto\ph(au)\]
is a ring homomorphism that we call the ring homomorphism associated to $\ph$ and we have
\[\ph(ab)=\Ph(a)\ph(b)\]
for all $a\in A$ and $b\in I$.
\end{thm}

\begin{proof}
This follows from \cite[Theorem 4.5]{bss} applied to the ``quadratic pseudomodule'' $M\cap I$ together with \cite[Lemma 4.9]{bss}.
\end{proof}

\subsection{Positive semidefinite matrices}

Fix $n\in\N_0$. Consider the subspace
\[S\R^{n\times n}:=\{A\in\R^{n\times n}\mid A=A^T\}\]
of symmetric matrices inside the real vector space $\R^{n\times n}$ of all real $n\times n$ matrices.
A matrix $A\in\R^{n\times n}$ is called \emph{positive semidefinite} if $A\in S\R^{n\times n}$ and
$x^TAx\ge0$ for all $x\in\R^n$. It is called \emph{positive definite} if $A\in S\R^{n\times n}$ and
$x^TAx>0$ for all $x\in\R^n\setminus\{0\}$. For matrices $A,B\in\R^{n\times n}$ we write $A\succeq B$ or $A\succ B$
to express that $A-B$ is positiv semidefinite or positive definite, respectively. A matrix $A\in\R^{n\times n}$
is called \emph{negative semidefinite} or \emph{negative definite} if $-A$ is \emph{positive semidefinite} or \emph{positive definite},
respectively. For $A,B\in\R^{n\times n}$ we write $A\preceq B$ and $A\prec B$ to express that $B\succeq A$ and
$B\succ A$, respectively.
Using the spectral theorem, one easily sees that
\[\R^{n\times n}_{\succeq0}:=\{A\in\R^{n\times n}\mid A\succeq0\}=\{B^TB\mid B\in\R^{n\times n}\}.\]
We equip $S\R^{n\times n}$ with the unique topology induced by any norm on $S\R^{n\times n}$.
It is clear that $\R^{n\times n}_{\succeq0}$ is a cone in $S\R^{n\times n}$ and one easily shows that its interior in $S\R^{n\times n}$ is
\[\R^{n\times n}_{\succ0}:=\{A\in\R^{n\times n}\mid A\succ0\}.\]

\subsection{Real closed fields}\label{subs:rcf}

We remind the reader of the basic facts about real closed fields that
we will need. A good reference is \cite[Section 8.3]{pd1}.

If $R$ is a field and $R^2=\{a^2\mid a\in R\}$ is the set of its squares (using the notation from Subsection \ref{subs:notation}),
then $R$ is called \emph{real closed} if
\[a\le b~:\iff b-a\in R^2\qquad(a,b\in R)\]
defines an order (i.e., a reflexive, transitive and antisymmetric relation) on the set $R$ with respect to which the intermediate
value theorem for polynomials holds: Whenever $f\in R[T]$ and $x,y\in R$ with $x<y$ and
$f(x)f(y)<0$, then there is $z\in R$ with $x<z<y$ and $f(z)=0$.

The prototype of all real closed fields is the field of real numbers $\R$.
In this article, all real closed fields we will consider will contain $\R$ as a subfield.

Fix such a real closed extension field $R$ of $\R$. Then
\[\O_R:=B_{(R,R_{\ge0})}=\{a\in\R\mid\exists N\in\N:-N\le a\le N\}\]
is a subring of $R$ containing $\R$ (note that $R_{\ge0}=R^2$ is a quadratic module of $R$)
and we call its elements the \emph{finite} elements of $R$. Moreover,
\[\m_R:=\left\{a\in\R\mid\forall N\in\N:-\frac1N\le a\le\frac1N\right\}\]
is an ideal of $\O_R$ and we call its elements the \emph{infinitesimal} elements of $R$.
The complement of $\m_R$ in $\O_R$ is the group of units of $\O_R$, i.e.,
\[\O_R^\times=\O_R\setminus\m_R\]
and thus $\m_R$ is the unique maximal ideal of $\O_R$. Using the completeness of the field of real numbers $\R$, one shows
easily that for every $a\in\O_R$, there is exactly one $\st(a)\in\R$, called
the \emph{standard part} of $a$, such that \[a-\st(a)\in\m_R.\]
The map $R\to\R,\ a\mapsto\st(a)$ is easily shown to be a ring homomorphism with
kernel $\m_R$. If $a,b\in\O_R$ with $a\le b$, then $\st(a)\le\st(b)$. Conversely, if $a,b\in\O_R$ satisfy $\st(a)<\st(b)$, then
$a<b$. The standard part $\st(p)$ of a polynomial $p\in\O_R[\x]$ arises
by replacing each coefficient of $p$ by its standard part. Also
\[\O_R[\x]\to\R[\x],\ p\mapsto\st(p)\] is a ring homomorphism. For $x=(x_1,\ldots,x_n)\in\O_R^n$, we set
\[\st(x):=(\st(x_1),\ldots,\st(x_n)).\]
It is easy to see that
\[R=\R\iff\O_R=\R\iff\m_R=\{0\}\]
and that for each $a\in R_{\ge0}$ there is a unique $b\in R_{\ge0}$ such that $b^2=a$ which one denotes by $\sqrt a$.
By Proposition \ref{bammodule}, we have
\[a\in\O_R\iff\sqrt a\in\O_R\]
for all $a\in\R_{\ge0}$ and one easily shows that
\[a\in\m_R\iff\sqrt a\in\m_R\]
for all $a\in R_{\ge0}$.
For $x\in R^n$, we set
\[\|x\|_2:=\sqrt{x_1^2+\ldots+x_n^2}\]
and observe that
\[\|x\|_2\in\O_R\iff x\in\O_R^n.\]

\subsection{Semialgebraic sets, real quantifier elimination and the finiteness theorem}\label{subs:tarski}
Let $R$ be a real closed extension field of $\R$ and $n\in\N_0$. Note for later
that $R^0=\{0\}$ where $0$ is the empty tuple $0=()$ and thus there are exactly two subsets of $R^0$, namely $\{0\}$ and $\emptyset$.
The Boolean algebra of all \emph{$\R$-semialgebraic subsets
of $R^n$} is the smallest set of subsets of $R^n$ that contains $\{x\in R^n\mid p(x)\ge0\}$ as an element
for each $p\in\R[\x]$ and that is closed under finite intersections (the empty intersection being defined as
$R^n$) and under complements (thus also under finite unions). It is easy to see that an equivalent definition would be that
an \emph{$\R$-semialgebraic subset of $R^n$} is a finite union of sets of the form
\[\{x\in R^n\mid g(x)=0,h_1(x)>0,\dots,h_k(x)>0\}\qquad(k\in\N_0,g,h_1,\dots,h_k\in\R[\x]).\]
We simply say \emph{semialgebraic subset of $\R^n$} instead of $\R$-semialgebraic subset of $\R^n$.

It is trivial that for each semialgebraic subset $S$ of $\R^n$, there is an $\R$-semialgebraic subset $S'$ of $R^n$ such that
$S'\cap\R^n=S$ (just use the same polynomials to define it) but it is a deep theorem that there is \emph{exactly one} such $S'$
which we denote by $S_R$ and which we call the \emph{transfer} of $S$ to $R$. This follows easily from the
nontrivial fact that each nonempty $\R$-semialgebraic subset of $R^n$ has a point in $\R^n$
\cite[Proposition~4.1.1]{bcr}. Having said this, it is now
trivial that the correspondence
\begin{align*}
S&\mapsto S_R\\
S\cap\R^n&\mapsfrom S
\end{align*}
defines an isomorphism between the Boolean algebras of semialgebraic subsets of $\R^n$ and the Boolean algebra
of $\R$-semialgebraic subsets of $R^n$, that is a bijection that respects finite intersections and complements. In particular, it preserves
the empty intersection, i.e., it maps $\R^n$ to $R^n$ and thus the empty set to the empty set. It also compatible with finite unions of
course.

By \emph{real quantifier elimination} \cite[Proposition~5.2.2]{bcr}, the following is true: If $n\in\N_0$ and $S$ is a semialgebraic subset of $\R^{n+1}$, then
\begin{align*}
S'&:=\{x\in\R^n\mid\forall y\in\R:(x,y)\in S\}\qquad\text{and}\\
S''&:=\{x\in\R^n\mid\exists y\in\R:(x,y)\in S\}
\end{align*}
are semialgebraic subsets of $\R^n$
and for all real closed extension fields $R$ of $\R$, we have
\begin{align*}
S'_R&=\{x\in R^n\mid\forall y\in R:(x,y)\in S_R\}\qquad\text{and}\\
S''_R&=\{x\in R^n\mid\exists y\in R:(x,y)\in S_R\}.
\end{align*}
This extends in the obvious way to finitely many quantifiers. In particular, if
$S$ is a semialgebraic subset of $\R^n$ and each of $Q_1,\ldots,Q_n$ stands for an universal or existential quantifier
$\forall$ or $\exists$, then consider the statements
\[(*)\qquad Q_1y_1\in\R:\ldots Q_ny_n\in\R:y\in S\]
and
\[(**)\qquad Q_1y_1\in R:\ldots Q_ny_n\in R:y\in S_R\]
and let $S'\subseteq\R^0=\{0\}$
denote the set of all $x\in\R^0$ for which $(*)$ holds, i.e.,
$S'=\{0\}$ if $(*)$ holds and $S'=\emptyset$ otherwise. Then $S_R'$ is the set of all $x\in R^0$
for which $(**)$ holds, i.e., $S'_R=\{0\}$ if $(*)$ holds and $S'_R=\emptyset$ otherwise. Since
$S'=\emptyset\iff S'_R=\emptyset$, we have
\[(*)\iff(**).\]
In this way, one can transfer certain statements from $\R$ to $R$ (or vice versa).
Since $S\mapsto S_R$ is an isomorphism of Boolean algebras as described above, one can do the same thing
if one deals with expressions similar to $(*)$ built up from finitely many atoms of the form ``$y\in S$'' ($S\subseteq\R^n$ a
semialgebraic set where the arity $n$ and the variables $y$ might be different each time) by the 
the logical connectives ``and'' and ``not'' (and thus also ``or'', ``\mbox{$\Longrightarrow$}'', ``$\Longleftrightarrow$'' etc.) and
quantifications over $\R$ (such as ``$\forall x\in\R$''). We call such expressions a \emph{formula} and unlike $(*)$ they might still contain \emph{free} variables
$x\in\R^n$ (i.e., variables that have not been quantified over). In first order logic, one would formalize this notion of formula
as first order formulas in the language of ordered fields with new constant symbols for each real number
\cite[Definition~2.2.3]{bcr}. Let $\Ps(x)$ stand
for such a formula with free variables $x\in\R^n$. Then it is clear that
\[\{x\in\R^n\mid\Ps(x)\}_R=\{x\in R^n\mid\Ps_R(x)\}\]
where $\Ps_R$ arises from $\Ps$ by replacing each $S$ by $S_R$ and each quantification over $\R$ by the corresponding
quantification over $R$ (e.g., ``$\forall y\in\R$'' by ``$\forall y\in R$'') \cite[Corollary 5.2.4]{bcr}.
In particular, if $\Ps$ has no free variables (that is $n=0$), then $\Ps$ holds if and only if $\Ps_R$ holds.
This is called the Tarski transfer principle \cite[Proposition~5.2.3]{bcr}.

Of utmost importance for us will be the
finiteness theorem from first order logic \cite[Theorem 1.5.6]{pd2} that says in particular:
If $(S_d)_{d\in\N}$ is a sequence of semialgebraic subsets of $\R^n$
such that \[S_1\subseteq S_2\subseteq S_3\subseteq\ldots\] and if \[\bigcup_{d\in\N} (S_d)_R=R^n\] for all real closed extension
fields $R$ of $\R$, then there exists $d\in\N$ such that \[S_d=\R^n.\] In this article, we will refer to this
nontrivial fact simply as the \emph{finiteness theorem}. It follows also easily from \cite[Theorem~2.2.11]{pd1}.
Note that the converse is trivial: If $d\in\N$ such that $S_d=\R^n$, then
even $(S_d)_R=R^n$ for all real closed extension fields $R$ of $\R$.

\begin{ex}\label{boundex}
There is some real closed extension field $R$ of $\R$ with $R\ne\R$. To prove this, we assume the contrary and we set
$S_d:=[-d,d]\subseteq\R$ for each $d\in\N$. Then $S_1\subseteq S_2\subseteq S_3\subseteq\ldots$ and
$\bigcup_{d\in\N}S_d=\R$ and thus of course
$\bigcup_{d\in\N} (S_d)_R=R$ for all real closed extension fields $R$ of $\R$ (since the only such is $\R$ by assumption).
By the finiteness theorem, we get $S_d=\R$ for some $d\in\N$, which is a contradiction.
\end{ex}

A concrete example for a proper real closed extension field of $\R$ is the field of real Puiseux series \cite[Example 1.2.3]{bcr} but we will not need this.

\section{Pure states and nonnegative polynomials over real closed fields}

Throughout this section, we let $R$ be a real closed extension field of $\R$ and we set
$\O:=\O_R$ and $\m:=\m_R$.

\subsection{The Archimedean property}

\begin{pro}\label{archmodulecharrcf}
Suppose $n\in\N_0$ and $M$ is a quadratic module of $\O[\x]$. Then the following are equivalent:
\begin{enumerate}[\rm(a)]
\item $M$ is Archimedean.
\item $\exists N\in\N:N-\sum_{i=1}^nX_i^2\in M$
\item $\exists N\in\N:\forall i\in\{1,\dots,n\}:N\pm X_i\in M$
\end{enumerate}
\end{pro}

\begin{proof}
It is trivial that (a) implies (b).

Now suppose that (b) holds. Then $N-X_i^2\in M$ and thus $X_i^2\in B_{(\O[\x],M)}$
for all $i\in\{1,\dots,n\}$. Apply now Proposition \ref{bammodule} to see that (c) holds.

That (c) implies (a) follows from Proposition \ref{bammodule} since $\O\subseteq B_{(\O[\x],M)}$.
\end{proof}

\begin{rem}
Looking at Proposition \ref{archmodulechar}(d), one could be inclined to think that one could add
\begin{multline*}
\exists m\in\N:\exists\ell_1,\dots,\ell_m\in M\cap\O[\x]_1:\exists N\in\N:\\
\emptyset\ne\{x\in R^n\mid\ell_1(x)\ge0,\dots,\ell_m(x)\ge0\}\subseteq[-N,N]_R^n
\end{multline*}
as another equivalent condition
in Proposition \ref{archmodulecharrcf}. Indeed, choose $R$ different from $\R$ (which is possible by
Example \ref{boundex}) and choose $\ep\in\m\setminus\{0\}$.
Then $\emptyset\ne\{0\}=\{x\in R\mid\ep x\ge0,-\ep x\ge0\}\subseteq[-1,1]_R$ but
the quadratic module \[\sum\O[X]^2+\sum\O[X]^2\ep X+\sum\O[X]^2(-\ep X)\]
generated by $\ep X$ and $-\ep X$
in $\O[X]$ is not Archimedean for if we had $N\in\N$ such that $N-X^2$ lies in it,
then taking standard parts would yield
$N-X^2\in\sum\R[X]^2$ which is obviously not possible.
\end{rem}

\subsection{The relevant ideals}

\begin{df}\label{adamshom}
For every $x\in\O^n$, we define $I_x$ to be the kernel of the ring homomorphism
\[\O[\x]\to\O,\ p\mapsto p(x).\]
\end{df}

\begin{pro}
Let $x\in\O^n$. Then $I_x=(X_1-x_1,\dots,X_n-x_n)$.
\end{pro}

\begin{proof}
It is trivial that $J:=(X_1-x_1,\dots,X_n-x_n)\subseteq I_x$.
Conversely, $p\equiv_Jp(x)=0$ for all
$p\in I_x$. This shows the converse inclusion $I_x\subseteq J$.
\end{proof}

\begin{lem}\label{coprime}
Suppose $x,y\in\O^n$ with $\st(x)\ne\st(y)$. Then $I_x$ and $I_y$ are coprime, i.e.,
$1\in I_x+I_y$.
\end{lem}

\begin{proof}
WLOG $x_1-y_1\notin\m$. Then $x_1-y_1\in\O^\times$ and
\[1=\frac{x_1-X_1}{x_1-y_1}+\frac{X_1-y_1}{x_1-y_1}\in I_x+I_y.\]
\end{proof}

Consider $x\in\O^n$.
We remind the reader that we denote by $I_x^2$ the second power of $I_x$ as an ideal, see
Subsection \ref{subs:notation}. It consist of those polynomials
that have a double zero at $x$, see Lemma \ref{membershipix} below. If one wants to show that such a polynomial
lies in a quadratic module of $\O[\x]$, it is often advantageous to intersect this quadratic module with $I_x^2$.
If $M$ is Archimedean, then $1$ is a unit for $M$ according to Definition \ref{defunit}
in the real vector space $\O[\x]$ by Proposition \ref{archmodulecharrcf}.
Since $1$ is not a member of $I_x^2$, it would be good to find a kind of replacement for it. This role will
be played by $u_x$ from the next lemma.

\begin{lem}\label{ixunit}
Let $M$ be an Archimedean quadratic module of $\O[\x]$ and $x\in\O^n$. Then
\[u_x:=(X_1-x_1)^2+\ldots+(X_n-x_n)^2\]
is a unit for $M\cap I_x^2$ in the real vector space $I_x^2$.
\end{lem}

\begin{proof}
Using the ring automorphism \[\O[\x]\to\O[\x],\ p\mapsto p(X_1-x_1,\ldots,X_n-x_n),\]
which is also an isomorphism of real vector spaces, we can reduce to the case $x=0$.
Since $u_x\in I_0^2$, it suffices to show that $I_0^2\subseteq B_{(\O[\x],M,u)}$.
Since $M$ is Archimedean, Proposition~\ref{bammodule} yields that
$B_{(\O[\x],M,u)}$ is an ideal of $\O[\x]$. Because of
\[I_0^2=(X_iX_j\mid i,j\in\{1,\dots,n\}),\]
it suffices therefore to show that $X_iX_j\in B_{(\O[\x],M,u)}$ for all $i,j\in\{1,\dots,n\}$.
Thus fix $i,j\in\{1,\dots,n\}$. Then $\frac12(X_i^2+X_j^2)\pm X_iX_j=\frac12(X_i\pm X_j)^2\in M$
and thus $\frac12u\pm X_iX_j\in M$. Since $u\in M$, this implies
$u\pm X_iX_j\in M$.
\end{proof}

\subsection{States over real closed fields}

\begin{nt}
We use the symbols $\nabla$ and $\hess$ to denote the gradient and the Hessian of a real-valued function of $n$ real variables, respectively. For a \emph{polynomial} $p\in\R[\x]$, we understand its
gradient $\nabla p$ as a column vector from $\R[\x]^n$, i.e., as a vector of polynomials. Similarly, its Hessian
$\hess p$ is a symmetric matrix polynomial of size $n$, i.e., a symmetric matrix from $\R[\x]^{n\times n}$.
Using formal partial derivatives, we more generally define $\nabla p\in R[\x]^n$ and
$\hess p\in R[\x]^{n\times n}$ even for $p\in R[\x]$.
\end{nt}

\begin{lem}\label{secondtypelemma}
Suppose $x\in\O^n$ and $\ph$ is a state of $(I_x^2,\sum\O[\x]^2\cap I_x^2,u_x)$ such that
$\ph|_{I_x^3}=0$. Then there exist
$v_1,\dots,v_n\in\R^n$ such that
$\sum_{i=1}^nv_i^Tv_i=1$ and
\[\ph(p)=\frac12\st\left(\sum_{i=1}^nv_i^T(\hess p)(x)v_i\right)\]
for all $p\in I_x^2$.
\end{lem}

\begin{proof}
As in the proof of Lemma \ref{ixunit}, one easily reduces to the case $x=0$.

\medskip
\textbf{Claim 1:} $\ph(au_x)=0$ for all $a\in\m$.

\smallskip
\emph{Explanation.} Let $a\in\m$. WLOG $a\ge0$. Then $a\in\O\cap R_{\ge0}=\O^2$ and thus
$au_x\in\sum\O[\x]^2\cap I_0^2$. This shows $\ph(au_x)\ge0$. It remains to show that $\ph(au_x)\le\frac1N$ for all
$N\in\N$. For this purpose, fix $N\in\N$. Then $\frac1N-a\in\O\cap R_{\ge0}=\O^2$ and thus
$\left(\frac1N-a\right)u_x\in\sum\O[\x]^2\cap I_0^2$. It follows that $\ph\left(\left(\frac1N-a\right)u_x\right)\ge0$, i.e.,
$\ph(au_x)\le\frac1N$.

\medskip
\textbf{Claim 2:} $\ph(aX_i^2)=0$ for all $a\in\m$ and $i\in\{1,\dots,n\}$.

\smallskip
\emph{Explanation.} Let $a\in\m$. WLOG $a\ge0$ and thus $a\in\O^2$.
Then \[\sum_{i=1}^n\underbrace{\ph(\overbrace{aX_i^2}^{\rlap{$\scriptstyle\in\O[\x]^2\cap I_0^2$}})}_{\ge0}=\ph(au_x)\overset{\text{Claim 1}}=0.\]

\medskip
\textbf{Claim 3:} $\ph(aX_iX_j)=0$ for all $a\in\m$ and $i,j\in\{1,\dots,n\}$.

\smallskip
\emph{Explanation.} Fix $i,j\in\{1,\dots,n\}$ and $a\in\m$. If $i=j$, then we are done by Claim 2. So suppose
$i\ne j$. WLOG $a\ge0$ and thus $a\in\O^2$.
Then \[a(X_i^2+X_j^2\pm 2X_iX_j)=a(X_i\pm X_j)^2\in\O[\x]^2\cap I_0^2\] and thus
$\pm 2\ph(aX_iX_j)\underset{\text{Claim 2}}=\ph(aX_i^2)+\ph(aX_j^2)\pm 2\ph(aX_iX_j)\ge 0$.

\medskip
\textbf{Claim 4:} $\ph(p)=\frac12\st\left(\tr\left((\hess p)(0)A\right)\right)$ for all $p\in I_0^2$ where
\[A:=\begin{pmatrix}\ph(X_1X_1)&\dots&\ph(X_1X_n)\\
\vdots&\ddots&\vdots\\
\ph(X_nX_1)&\dots&\ph(X_nX_n)
\end{pmatrix}.
\]

\smallskip
\emph{Explanation.} Let $p\in I_0^2$. By $\ph|_{I_0^3}=0$, we can reduce to the case
$p=aX_iX_j$ with $i,j\in\{1,\dots,n\}$ and $a\in\O$.
Using  Claim 3, we can assume $a=1$. Comparing both sides,
yields the result.

\medskip
\textbf{Claim 5:} $A\succeq0$

\smallskip
\emph{Explanation.} If $x\in\R^n$, then $x^TAx=\ph((x_1X_1+\ldots+x_nX_n)^2)\ge0$ since
\[(x_1X_1+\ldots+x_nX_n)^2\in\R[\x]^2\cap I_0^2\subseteq\sum\O[\x]^2\cap I_0^2.\]
By Claim 5, we can choose $B\in\R^{n\times n}$ such that $A=B^TB$. Denote by $v_i$ the $i$-th
row of $B$ for $i\in\{1,\dots,n\}$. Then by Claim 4, we get
\begin{multline*}
\ph(p)=\frac12\st(\tr((\hess p)(0)A))
=\frac12\st(\tr((\hess p)(0)B^TB))\\
=\frac12\st(\tr(B(\hess p)(0)B^T))=\frac12\st\left(\sum_{i=1}^nv_i^T(\hess p)(0)v_i\right)
\end{multline*}
for all $p\in I_0^2$.
In particular, we obtain $1=\ph(u_x)=\sum_{i=1}^nv_i^Tv_i$.
\end{proof}

\begin{lem}\label{stpointev}
Let $\Ph\colon\O[\x]\to\R$ be a ring homomorphism. Then there is some $x\in\R^n$ such that
$\Ph(p)=\st(p(x))$ for all $p\in\O[\x]$.
\end{lem}

\begin{proof}
Being a ring homomorphism, $\Ph$ maps each rational number to itself (since it maps $1$ to $1$) and
each square to a square. It follows that $\Ph|_\R$ is a monotonic function which fixes $\Q$ pointwise.
This easily implies $\Ph|_\R=\id_\R$. It is also easy to show that $\Ph|_\m=0$.
Indeed, for each
$N\in\N$ and $a\in\m$, we have $\frac1N\pm a\in R_{\ge0}\cap\O=\O^2$
and therefore
$\frac1N\pm\Ph(a)\in\R_{\ge0}$. 
Finally set \[x:=(\Ph(X_1),\dots,\Ph(X_n))\in\R^n\]
and use that $\Ph|_\R=\id_\R$, $\Ph|_\m=0$ and that $\Ph$ is a ring homomorphism.
\end{proof}

The following result can be seen as a dichotomy for pure states. In the situation described in the theorem, there are two
types of states: Up to a certain harmless scaling, the first type is just evaluation at a point different from the $\st(x_i)$ and taking the
standard part. The second type of pure states is, again up to a scaling, taking the standard part of
a mixture of second directional derivatives at one of the $\st(x_i)$. It shows that the state space plays roughly the role of a
``blowup'' of affine space in a certain sense. The
$\st(x_i)$ are removed and, ignoring the subtle issue with the mixtures, second directional derivatives at the $\st(x_i)$ are thrown in instead. In this way each $\st(x_i)$ is morally replaced by a projective space (note that second
directional derivatives do not distinguish opposite directions).

If one does \emph{not} ignore the scaling issue, then another interesting interpretation is that type (2) states can be seen,
in a certain sense, as mixtures
of limits of type (1) states. We leave this interpretation as an exercise to the interested reader.

In the case $R=\R$, a less concrete variant of this theorem is \cite[Corollary 4.12]{bss}. What is essentially new here is the case
$R\ne\R$ in which the standard part map turns out to play a big role. Note also the subtle issue that we suppose the
\emph{standard parts of the} $x_i$ to be pairwise different and not just the $x_i$. The theorem is the main step for Theorem \ref{mainrep} below, on which in turn all our main results will be based. It will not be used elsewhere in this article.

\begin{thm}\label{dicho}
Let $M$ be an Archimedean quadratic module of $\O[\x]$ and set
\[S:=\{x\in\R^n\mid\forall p\in M:\st(p(x))\ge0\}.\]
Moreover, suppose $k\in\N_0$ and let $x_1,\ldots,x_k\in\O^n$ satisfy $\st(x_i)\ne\st(x_j)$ for
$i,j\in\{1,\dots,k\}$ with $i\ne j$. Then
$u:=u_{x_1}\dotsm u_{x_k}$ is a unit for $M\cap I$ in
\[I:=I_{x_1}^2\dotsm I_{x_k}^2=I_{x_1}^2\cap\ldots\cap I_{x_k}^2\]
and for all pure states $\ph$ of $(I,M\cap I,u)$ (where $I$ is understood as a real vector space),
exactly one of the following cases occurs:
\begin{enumerate}[\rm(1)]
\item There is an $x\in S\setminus\{\st(x_1),\dots,\st(x_k)\}$ such that
\[\ph(p)=\st\left(\frac{p(x)}{u(x)}\right)\]
for all $p\in I$.
\item There is an $i\in\{1,\dots,k\}$ and $v_1,\dots,v_n\in\R^n$ such that
$\sum_{\ell=1}^nv_\ell^Tv_\ell=1$ and
\[\ph(p)=\st\left(\frac{\sum_{\ell=1}^nv_\ell^T(\hess p)(x_i)v_\ell}{2\prod_{\substack{j=1\\j\ne i}}^ku_{x_j}(x_i)}\right)\]
for all $p\in I$.
\end{enumerate}
\end{thm}

\begin{proof}
The Chinese remainder theorem from commutative algebra shows that
\[I=I_{x_1}^2\dotsm I_{x_k}^2=I_{x_1}^2\cap\ldots\cap I_{x_k}^2\]
since $I_{x_i}$ and $I_{x_j}$ and thus also $I_{x_i}^2$ and $I_{x_j}^2$ are coprime for all
$i,j\in\{1,\dots,k\}$ with $i\ne j$.
By Lemma \ref{ixunit}, $u_{x_i}$ is a unit for $M\cap I_{x_i}^2$ in $I_{x_i}^2$
for each $i\in\{1,\dots,k\}$. To show that $u$ is a unit
for the cone $M\cap I$ in the real vector space $I$, it suffices to find for all
$a_1,b_1\in I_{x_1},\dots,a_k,b_k\in I_{x_k}$ an $N\in\N$ such that
$Nu+ab\in M$ where we set $a:=a_1\dotsm a_k$ and $b:=b_1\dotsm b_k$.
Because of $Nu+ab=(Nu-\frac12a^2-\frac12b^2)+\frac12(a+b)^2$, it is enough to find $N\in\N$ with
$Nu-a^2\in M$ and $Nu-b^2\in M$. By symmetry, it suffices to find $N\in\N$ with
$Nu-a^2\in M$. Choose $N_i\in\N$ with $N_iu_{x_i}-a_i^2\in M$ for $i\in\N$. We now claim that
$N:=N_1\dotsm N_k$ does the job. Indeed, the reader shows easily by induction that actually
\[N_1\dotsm N_iu_{x_1}\dotsm u_{x_i}-a_1^2\dotsm a_i^2\in M\]
for $i\in\{1,\dots,k\}$. Now let $\ph$ be a pure state of $(I,M\cap I,u)$.
By Theorem \ref{dichotomy},
\[\Ph\colon\O[\x]\to\R,\ p\mapsto\ph(pu)\] is a ring homomorphism 
and we have
\[(*)\qquad \ph(pq)=\Ph(p)\ph(q)\]
for all $p\in\O[\x]$ and $q\in I$.
By Lemma \ref{stpointev}, we can choose $x\in\R^n$ such that
\[\Ph(p)=\st(p(x))\]
for all $p\in\O[\x]$.
Since $u\in I\cap\sum\O[\x]^2$, we have
\[\st(p(x))=\Ph(p)=\Ph(p)\ph(u)\overset{(*)}=\ph(pu)=\ph(up)\overset{up\in M}\in\ph(M)\subseteq\R_{\ge0}\]
for all $p\in M$ which means $x\in S$.

\smallskip
Now first suppose that $\Ph(u)\ne0$.
Then $\st(u_{x_i}(x))\ne0$ and therefore $\st(x)\ne\st(x_i)$ for all $i\in\{1,\dots,k\}$. Moreover,
\[\st(p(x))=\Ph(p)=\ph(pu)=\ph(up)\overset{(*)}=\Ph(u)\ph(p)=\st(u(x))\ph(p)\] for all $p\in I$.
Thus, (1) occurs.

\smallskip
Now suppose that $\Ph(u)=0$. We show that then (2) occurs.
Due to
\[\prod_{i=1}^k\Ph(u_{x_i})=
\Ph(u)=0,\]
 we can choose $i\in\{1,\dots,k\}$ such that $\st(u_{x_i}(x))=\Ph(u_{x_i})=0$.
Then $x=\st(x_i)$. Define
\[\ps\colon I_{x_i}^2\to\R,\ p\mapsto\ph\left(p\prod_{\substack{j=1\\j\ne i}}^ku_{x_j}\right).
\]
Since $u_{x_j}\in\sum\O[\x]^2\cap I_{x_j}^2$ for all $j\in\{1,\dots,k\}$, it follows that
$\ps\in S(I_{x_i}^2,M\cap I_{x_i}^2,u_{x_i})$. If $p\in I_{x_i}$ and $q\in I_{x_i}^2$, then
\[\ps(pq)=\ph\left(pq\prod_{\substack{j=1\\j\ne i}}^ku_{x_j}\right)
\overset{(*)}=
\Ph(p)\ph\left(q\prod_{\substack{j=1\\j\ne i}}^ku_{x_j}\right)=0\]
since $\Ph(p)=\st(p(x))=(\st(p))(x)=(\st(p))(\st(x_i))=\st(p(x_i))=\st(0)=0$.
It follows that $\ps|_{I_{x_i}^3}=0$.
We can thus apply Lemma \ref{secondtypelemma} to $\ps$ and obtain
$v_1,\dots,v_n\in\R^n$ such that
$\sum_{\ell=1}^nv_\ell^Tv_\ell=1$ and
\[\ps(p)=\frac12\st\left(\sum_{\ell=1}^nv_\ell^T(\hess p)(x_i)v_\ell\right)\]
for all $p\in I_{x_i}^2$.
Because of $\st(x_i)\ne\st(x_j)$ for $j\in\{1,\dots,k\}\setminus\{i\}$, we have
\[c:=\Ph\left(\prod_{\substack{j=1\\j\ne i}}^k u_{x_i}\right)=\prod_{\substack{j=1\\j\ne i}}^k\Ph(u_{x_i})=
\prod_{\substack{j=1\\j\ne i}}^k(\st(u_{x_j}))(\st(x_i))\ne0.\]
Hence we obtain
\[c\ph(p)\overset{(*)}=\ps(p)\]
for all $p\in I$.

\smallskip
It only remains to show that (1) and (2) cannot occur both at the same time. If (1) holds, then
we have obviously $\ph(u^2)\ne0$. If (2) holds, then $\ph(u^2)=0$ since $\hess(u^2)(x_i)=0$ for
all $i\in\{1,\dots,k\}$ as one easily shows.
\end{proof}

\subsection{Nonnegative polynomials over real closed fields}\label{subs:nnrcf}

In this subsection, we collect the consequences for sums of squares representations of nonnegative polynomials that
will be important for our applications to both: systems of polynomial inequalities and polynomial optimization problems.
We begin with a remark that shows that $I_x^2$ consist of those polynomials over $\O$ that have a \emph{double zero} at $x$.

\begin{lem}\label{membershipix}
For all $x\in\O^n$, we have
\[I_x^2=\left\{p\in\O[\x]\mid p(x)=0,\nabla p(x)=0\right\}.\]
\end{lem}

\begin{proof}
For $x=0$ it is easy. One reduces the general case to the case $x=0$ as in the proof of
Lemma~\ref{ixunit}.
\end{proof}

In the following key theorem, note that the positivity conditions on $f$ actually depend just on the standard part $\st(f)$
of $f$ (see Subsection \ref{subs:rcf}) but it will be essential to us that $f$ needs to vanish doubly not
necessarily \emph{at} the standard parts of the $x_i$ but only infinitesimally \emph{nearby}.
This key theorem will be used directly in both Theorem \ref{putinarzerosdegreebound} which has
implications on polynomial optimization and Theorem \ref{linearstability} which is important for solving systems of polynomial
inequalities.

\begin{thm}\label{mainrep}
Let $M$ be an Archimedean quadratic module of $\O[\x]$ and set
\[S:=\{x\in\R^n\mid\forall p\in M:\st(p(x))\ge0\}.\]
Moreover, suppose $k\in\N_0$ and let $x_1,\ldots,x_k\in\O^n$ have pairwise distinct standard parts. Let \[f\in\bigcap_{i=1}^kI_{x_i}^2\]
such that \[\st(f(x))>0\] for all $x\in S\setminus\{\st(x_1),\ldots,\st(x_k)\}$ and
\[\st(v^T(\hess f)(x_i)v)>0\] for all $i\in\{1,\dots,k\}$ and $v\in\R^n\setminus\{0\}$.
Then $f\in M$.
\end{thm}

\begin{proof}
Define $I$ and $u$ as in Theorem \ref{dicho}. By Lemma \ref{membershipix}, we have $f\in I$.
We will apply Corollary \ref{conemembershipextr} to the real vector space $I$, the cone $M\cap I$ in $I$ and
the unit $u$ for $M\cap I$. From Theorem \ref{dicho}, we see easily that $\ph(f)>0$ for all
$\ph\in\extr S(I,M\cap I,u)$.
\end{proof}

\begin{cor}\label{mainrep2}
Let $M$ be an Archimedean quadratic module of $\O[\x]$ and set
\[S:=\{x\in\R^n\mid\forall p\in M:\st(p(x))\ge0\}.\]
Moreover, let $k\in\N_0$ and $x_1,\ldots,x_k\in\O^n$ such that their standard parts
are pairwise distinct and lie in the interior of $S$. 
Let $f\in\O[\x]$ such that
\[f\in\bigcap_{i=1}^kI_{x_i}^2.\]
Set again $u:=u_{x_1}\dotsm u_{x_k}\in\O[\x]$.
Suppose there is $\ep\in\R_{>0}$ such
that \[f\ge\ep u\text{ on }S.\]
Then $f\in M$.
\end{cor}

\begin{proof}
By Theorem \ref{mainrep}, we have to show:
\begin{enumerate}[(a)]
\item $\forall x\in S\setminus\{\st(x_1),\ldots,\st(x_k)\}:\st(f(x))>0$
\item $\forall i\in\{1,\dots,k\}:\forall v\in\R^n\setminus\{0\}:\st(v^T(\hess f)(x_i)v)>0$
\end{enumerate}
It is easy to show (a). To show (b), fix $i\in\{1,\ldots,k\}$.
Because of $f-\ep u\ge0$ on $S$ and \[(f-\ep u)(x_i)=f(x_i)-\ep u(x_i)=0-0=0,\]
$\st(x_i)$ is a local minimum of $\st(f-\ep u)\in\R[\x]$ on $\R^n$. From
elementary analysis, we know therefore that $(\hess\st(f-\ep u))(\st(x_i))\succeq0$.
Because of $u_{x_i}(x_i)=0$ and $\nabla u_{x_i}(x_i)=0$, we get
\[\hess u(x_i)=
\left(\prod_{\substack{j=1\\j\ne i}}^ku_{x_j}(x_i)\right)\hess u_{x_i}(x_i)=
2\left(\prod_{\substack{j=1\\j\ne i}}^ku_{x_j}(x_i)\right)I_n.
\]
Therefore
\[\st(v^T(\hess f)(x_i)v)\ge\ep\st(v^T(\hess u)(x_i)v)=2\ep v^Tv\st\left(\prod_{\substack{j=1\\j\ne i}}^ku_{x_j}(x_i)\right)>0\]
for all $v\in\R^n\setminus\{0\}$.
\end{proof}

\begin{cor}\label{mainrep3}
Let $n,m\in\N_0$ and let $\g\in\R[\x]^m$ such that $M(\g)$ is Archimedean.
Moreover, let $k\in\N_0$ and $x_1,\ldots,x_k\in\O^n$ and $\ep\in\R_{>0}$ such that
the sets
\[x_1+(B_\ep(0))_R,\ldots,x_k+(B_\ep(0))_R\]
are pairwise disjoint and all contained in $S(\g)_R$.
Set $u:=u_{x_1}\dotsm u_{x_k}\in\O[\x]$.
Let $f\in\O[\x]$ such that $f\ge\ep u$ on $S(\g)$ and
\[f(x_1)=\ldots=f(x_k)=0.\]
Then $f$ lies in the quadratic module generated by $\g$ in $\O[\x]$.
\end{cor}

\begin{proof}
This follows easily from Corollary \ref{mainrep2} once we show that \[\nabla f(x_1)=\ldots=\nabla f(x_k)=0.\]
Choose $d\in\N_0$ with $f\in R[\x]_d$.
Since $f\ge\ep u\ge0$ on $S$ and thus $f\ge0$ on $x_i+(B_\ep(0))_R$ for all $i\in\{1,\ldots,k\}$, it suffices
to prove the following: If $p\in R[\x]_d$, $x\in R^n$, $\de\in R_{>0}$ such that $p\ge0$ on $x+(B_\de(0))_R$ and
$p(x)=0$, then $\nabla p(x)=0$. To see this, we employ the Tarski transfer principle:  One can formulate as a formula
that $\nabla p(x)=0$ holds for all $p\in\R[\x]_d$, $x\in\R^n$ and $\de\in\R_{>0}$ that satisfy $p\ge0$ on
$x+(B_\de(0))_R=B_\de(x)$ and $p(x)=0$.
Since this holds true by elementary calculus, we can transfer it to the real closed field $R$.
\end{proof}

\section{Membership in truncated quadratic modules}

In this section, we study the implications of Subsection \ref{subs:nnrcf} to sums of squares representations of
\emph{arbitrary degree} polynomials. In Subsection \ref{subs:cons-pop}, we will see the implications to polynomial optimization. 

\subsection{Degree bounds for Scheiderer's generalization of Putinar's Positivstellensatz}

Scheiderer gave far-reaching generalizations of Putinar's Positivstellensatz to nonnegative polynomials with zeros
\cite{s1,mar}. An important concrete instance of this is Corollary \ref{putinarzeros} below (Putinar's theorem
\cite[Lemma 4.1]{put} corresponds to the special case $k=0$ of that corollary). On the other hand, Prestel \cite{pre} proved
the existence of degree bounds for Putinar's theorem, more precisely he proved the statement of
Corollary \ref{putinardegreebound} below (this is not literally stated in his article but follows clearly from \cite[Section 4]{pre},
see also \cite[Chapter 8]{pd1}).
The next theorem is a common generalization of both:

\begin{thm}
\label{putinarzerosdegreebound}
Let $n,m\in\N_0$ and $\g\in\R[\x]^m$ such that $M(\g)$ is Archimedean.
Moreover, let $k\in\N_0$, $N\in\N$ and $\ep\in\R_{>0}$.
Then there exists \[d\in\N_0\] such that for all $f\in\R[\x]_N$ with all coefficients in $[-N,N]$, we have for
\[\{x\in S(\g)\mid f(x)=0\}=\{x_1,\ldots,x_k\}:\]
If the balls $B_{\ep}(x_1),\ldots,B_\ep(x_n)$ are pairwise disjoint and contained in $S(\g)$ and if we have
$f\ge\ep u$ on $S(\g)$
where $u:=u_{x_1}\dotsm u_{x_k}\in\R[\x]$, then \[f\in M_d(\g).\]
\end{thm}

\begin{proof}
Set $\nu:=\dim\R[\x]_N$. For each $d\in\N_0$, define the set $S_d\subseteq\R^\nu$
of all $a\in \R^\nu$ such that the following holds:
If $a\in[-N,N]^\nu$ and if
$a$ is the vector of coefficients (in a certain fixed order) of a polynomial $f\in\R[\x]_N$ with
exactly $k$ zeros $x_1,\dots,x_k$ on $S$, then
at least one of the following conditions (a), (b) and (c) is fulfilled:
\begin{enumerate}[(a)]
\item The sets $B_{\ep}(x_1),\ldots,B_\ep(x_n)$ are not pairwise disjoint or not all contained in $S$.
\item $f\ge\ep u$ on $S$ is violated where $u:=u_{x_1}\dotsm u_{x_k}\in\R[\x]$.
\item $f$ is not a sum of $d$ elements from $\R[\x]$ where each term in the sum is of degree at most $d$
and is of the form $p^2g_i$ with $p\in\R[\x]$ and $i\in\{0,\dots,m\}$ where $g_0:=1\in\R[\x]$.
\end{enumerate}
It is obvious that $S_0\subseteq S_1\subseteq S_2\subseteq\ldots$ and by real quantifier elimination
$S_d$ is a semialgebraic set for each $d\in\N_0$, see Subsection \ref{subs:tarski} above.
By Corollary \ref{mainrep3},
we have $\bigcup_{d\in\N_0}(S_d)_R=R^\nu$ for each real closed extension field $R$ of $\R$. By the finiteness theorem
(see Subsection \ref{subs:tarski} above),
it follows that $S_d=\R^\nu$ for some $d\in\N_0$ which is what we have to show.
\end{proof}

If we specialize the preceding theorem by setting $k=0$, we now obtain in the following corollary
Prestel's degree bound for Putinar's Positivstellensatz.
If we specialize also \emph{the proof of} the theorem by setting $k=0$, then it simplifies dramatically but
it is still considerably different from Prestel's proof.
What Prestel and we have in common is however the technique of working with non-Archimedean real closed fields
instead of the reals which enables us to work with infinitesimal neighborhoods and prevents us at many places
from having to quantify the size of certain neighborhoods of points. 
Nevertheless a quantitative version of the following corollary has also been proven with completely different means
by Nie and the second author \cite[Theorem 6]{ns}. To extend their quantitative approach in order to
try to prove a quantitative version of the above theorem would seem to be a very ambitious project that might be too tedious
even to begin with, at least in the general case.

\begin{cor}[Prestel]\label{putinardegreebound}
Let $n,m\in\N_0$ and $\g\in\R[\x]^m$ such that $M(\g)$ is Archi\-medean.
Moreover, let $N\in\N$ and $\ep\in\R_{>0}$.
Then there exists \[d\in\N_0\] such that for all $f\in\R[\x]_N$ with all coefficients in $[-N,N]$
and with $f\ge\ep$ on $S(\g)$, we have \[f\in M_d(\g).\]
\end{cor}

\begin{pro}\label{squeeze}
Suppose $S\subseteq\R^n$ is compact, $x_1,\dots,x_k\in S^\circ$ are pairwise distinct,
$u:=u_{x_1}\dotsm u_{x_k}\in\R[\x]$ and $f\in\R[\x]$ with $f(x_1)=\ldots=f(x_k)=0$.
Then the following are equivalent:
\begin{enumerate}[\rm(a)]
\item $f>0$ on $S\setminus\{x_1,\ldots,x_k\}$ and $\hess f(x_1),\dots,\hess f(x_k)$ are positive definite.
\item There is some $\ep\in\R_{>0}$ such that $f\ge\ep u$ on $S$.
\end{enumerate}
\end{pro}

\begin{proof}
It is easy to show that (b) implies (a), cf. the proof of Corollary \ref{mainrep2}.

Conversely, suppose that (a) holds. We show (b). It is easy to show that one can WLOG
assume that \[S=\bigdotcup_{i=1}^k\overline{B_\ep(x_i)}\] for some $\ep>0$.
Then one finds easily an Archimedean quadratic module $M$ of $\R[\x]$ such that
\[S=\{x\in\R^n\mid\forall p\in M:p(x)\ge0\}.\]
A strengthened version of Theorem \ref{mainrep} now yields $f-\ep u\in M$ for some $\ep\in\R_{>0}$
and thus $f-\ep u\ge0$ on $S$. One gets this strengthened version of Theorem \ref{mainrep}
by applying (a)$\implies$(c) from Theorem \ref{conemembershipunitextreme} instead of
Corollary \ref{conemembershipextr} in its proof. Alternatively, we leave it as an exercise to the reader to
give a direct proof using only basic multivariate analysis.
\end{proof}

As a consequence, we obtain Scheiderer's well-known generalization of Putinar's Positivstellensatz to nonnegative polynomials
with finitely many zeros \cite[Corollary 3.6]{s1} that was mentioned in Subsections \ref{subs:our} and \ref{subs:pospol} above.

\begin{cor}\label{putinarzeros}
Let $\g\in\R[\x]^m$ such that $M(\g)$ is Archimedean.
Moreover, suppose $k\in\N_0$ and $x_1,\ldots,x_k\in S(\g)^\circ$ are pairwise distinct.
Let $f\in\R[\x]$ such that $f(x_1)=\ldots=f(x_k)=0$, $f>0$ on $S\setminus\{x_1,\ldots,x_k\}$ and
$\hess f(x_1),\dots,\hess f(x_k)$ are positive definite. Then $f\in M(\g)$.
\end{cor}

\begin{proof}
This follows from Theorem \ref{putinarzerosdegreebound} by Proposition \ref{squeeze}.
\end{proof}

\begin{rem}
Because of Proposition \ref{squeeze}, Theorem \ref{putinarzerosdegreebound} is really a
quantitative version of Corollary \ref{putinarzeros}.
\end{rem}

\begin{rem}\label{commentonbounds}
\begin{enumerate}[(a)]
\item In Condition (c) from the proof of Theorem \ref{putinarzerosdegreebound}, we speak of
``a sum of $d$ elements'' instead of ``a sum of elements'' (which would in general
be strictly weaker).
Our motivation to do this was that this is the easiest way to make sure that we can formulate
(c) in a formula in the sense of Subsection \ref{subs:tarski}. A second motivation could have been to formulate
Theorem \ref{putinarzerosdegreebound} in stronger way, namely by letting $d$ be a bound
not only on the degree of the quadratic module representation but also on the number of terms in it.
This second motivation is however not interesting because we get also from the Gram matrix method
\cite[Theorem~8.1.3]{pd1} a bound on this number of terms (a priori bigger than $d$ but after readjusting $d$ we can
again assume it to be $d$). We could have used the Gram matrix method already to see that
``a sum of elements'' (instead of ``a sum of $d$ elements'') can also be expressed in a formula.
\item We could strengthen condition (c) from the proof of
Theorem \ref{putinarzerosdegreebound}, by writing ``with $p\in R[\x]$ all of whose coefficients lie
in $[-d,d]_R$'' instead of just ``with $p\in R[\x]$''. Then $\bigcup_{d\in\N_0}(S_d)_R=R^\nu$
would still hold for all real closed extension fields $R$ of $\R$ since Corollary \ref{mainrep3} states that
$f$ lies in the quadratic module generated by $\g$ even in $\O[\x]$ not just in $\R[\x]$.
This would lead to a real strengthening of Theorem \ref{putinarzerosdegreebound}, namely we
could ensure that $d$ is a bound
not only on the degree of the quadratic module representation but also on the size of the coefficients
in it. However, we do currently not know of any application of this and therefore renounced to carry this
out.
\end{enumerate}
\end{rem}

\subsection{Consequences for the Lasserre relaxation of a polynomial optimization problem}\label{subs:cons-pop}

We now harvest our first crop and present our
application to polynomial optimization that we have announced and discussed already in Subsection
\ref{subs:our}.

\begin{thm}\label{optimizationbound}
Let $n,m\in\N_0$ and $\g\in\R[\x]^m$ such that $M(\g)$ is Archimedean and $S(\g)\ne\emptyset$.
Moreover, let $k\in\N$, $N\in\N$ and $\ep\in\R_{>0}$.
Then there exists \[d\in\N_0\] such that for all $f\in\R[\x]_N$ with all coefficients in $[-N,N]$,
\[a:=\min\{f(x)\mid x\in S(\g)\}\quad\text{and}\quad\{x\in S(\g)\mid f(x)=a\}=\{x_1,\ldots,x_k\},\]
we have: If the balls $B_\ep(x_1),\ldots,B_\ep(x_n)$ are pairwise disjoint and contained in $S(\g)$ and if we have
\[f\ge a+\ep u\text{ on }S(\g)\]
where $u:=u_{x_1}\dotsm u_{x_k}\in\R[\x]$, then $f-a\in M(\g)$ and consequently
\[\lasserre_d(f,\g)=a.\]
\end{thm}

\begin{proof}
Set $c:=N\left(\dim\R[\x]_N\right)\max_{x\in S(\g)}\|x\|\in\R_{\ge0}$.
We apply Theorem \ref{putinarzerosdegreebound} to the given data with $N$ replaced by some $\hat N\in\N$
satisfying $\hat N\ge N+c$,
and choose the corresponding $d\in\N_0$. Now let $f\in\R[\x]_N$ with all coefficients in $[-N,N]$,
\[a:=\min\{f(x)\mid x\in S(\g)\}\quad\text{and}\quad\{x\in S(\g)\mid f(x)=a\}=\{x_1,\ldots,x_k\}.\]
By the choice of $\widehat N$, one shows easily that all values of $f$ on $S(\g)$ lie in the interval $[-c,c]$. In particular,
$a\in[-c,c]$. Therefore all coefficients of \[\widehat f:=f-a\in\R[\x]_N\subseteq\R[\x]_{\widehat N}\]
lie in $[-\widehat N,\widehat N]$. By choice of $a$, we have of course
\[\{x\in S(\g)\mid f(x)=a\}=\{x_1,\ldots,x_k\}.\]
The rest is now easy using Theorem \ref{putinarzerosdegreebound}.
\end{proof}

\section{Linear polynomials and truncated quadratic modules}

In this section, we study the implications of Subsection \ref{subs:nnrcf} for membership
of \emph{linear} polynomials
in truncated quadratic modules. It will not be sufficient to build upon the results of the preceding section where we studied
membership
of \emph{arbitrary degree} polynomials
in truncated quadratic modules. Instead we will rather have to resort again
to the results about membership in quadratic modules over real closed fields that we have obtained in
Subsection \ref{subs:nnrcf}. Our ultimate goal are the applications to, in a broad sense,
solving systems of polynomial inequalities that we will harvest in Subsection \ref{subs:cons-solving}.

\subsection{Concavity}

\begin{df}\label{dfconcave}
Let $g\in\R[\x]$. If $x\in\R^n$, then we call $g$
\emph{strictly concave} at $x$ if $(\hess g)(x)\prec0$
and \emph{strictly quasiconcave} at $x$ if \[((\nabla g)(x))^Tv=0\implies
v^T(\hess g)(x)v<0\]
for all $v\in\R^n\setminus\{0\}$. If $S\subseteq\R^n$,
we call $g$ \emph{strictly (quasi-)concave} on $S$ is $g$ is
\emph{strictly (quasi-)concave} at every point of $S$.
\end{df}

\begin{rem}\label{whatmeansquasiconcavity}
It is trivial that quasiconcavity of a polynomial $g$ at $x$ depends only on the function $U\to\R,\ x\mapsto g(x)$
where $U$ is an arbitrarily small neighborhood of $x$. But if $g(x)=0$ and
$\nabla g(x)\ne0$, then it actually
depends only on the function
\[U\to\{-1,0,1\},\ x\mapsto\sgn(g(x)).\]
It then roughly means that the hypersurface $S(g)=\{x\in\R^n\mid g(x)\ge0\}$ looks locally around $x$ almost like a ``boundary piece''
of a closed disk \cite[Proposition~3.4]{ks}.
In the exceptional case where
$g(x)=0$ and $\nabla g(x)=0$, $g$ is strictly quasiconcave at $x$ if and only if it is strictly concave at $x$, and in this case
$S(g)$ equals locally the singleton set $\{x\}$.
\end{rem}

The following remark essentially appears in \cite[Lemma 11]{hn1}.

\begin{rem}\label{quasiconcavelift}
For $g\in\R[\x]$ and $x\in\R^n$, the following are obviously equivalent:
\begin{enumerate}[(a)]
\item $g$ is strictly quasiconcave at $x$.
\item $\exists\la\in\R:\la(\nabla g(x))(\nabla g(x))^T\succ(\hess g)(x)$
\end{enumerate}
\end{rem}

\begin{lem}\label{strictly2neighbor}
Let $g\in\R[\x]$. If $g$ is strictly (quasi-)concave at $x\in\R^n$,
then there is a neighborhood $U$ of $x$ such that $g$ is
strictly (quasi-)concave on $U$.
\end{lem}

\begin{proof}
The first statement follows from the openness of $\R^{n\times n}_{\succ0}$
by the continuity of $\R^n\to S\R^{n\times n},\ x\mapsto(\hess g)(x)$. The second statement follows
similarly by using the equivalence of (a) and (b) in Remark \ref{quasiconcavelift}.
\end{proof}

\begin{rem}\label{derexo}
Let $g\in\R[\x]$ and $x\in\R^n$ with $g(x)=0$. Then
\begin{align*}
(\nabla(g(1-g)^k))(x)&=(\nabla g)(x)\qquad\text{and}\\
(\hess(g(1-g)^k))(x)&=(\hess g-2k(\nabla g)(\nabla g)^T)(x).
\end{align*}
\end{rem}

\begin{lem}\label{quasi2concave}
Suppose $g\in\R[\x]$, $u\in\R^n$ and $g(u)=0$.
Then the following are equivalent:
\begin{enumerate}[(a)]
\item $g$ is strictly quasiconcave at $u$.
\item There exists $k\in\N$ such that $g(1-g)^k$ is strictly concave at $u$.
\item There exists $k\in\N$ such that for all $\ell\in\N$ with $\ell\ge k$, we have
that $g(1-g)^\ell$ is strictly concave at $u$.
\end{enumerate}
\end{lem}

\begin{proof}
Combine Remarks \ref{derexo} and \ref{quasiconcavelift}.
\end{proof}

\subsection{Real Lagrange multipliers}

\begin{rem}\label{gradientpositive}
If $u,x\in\R^n$ with $v:=x-u\ne0$, $g\in\R[\x]$, $g$ is quasiconcave at $u$, $g(u)=0$ and $g\ge0$ on $\conv\{u,x\}$, then
obviously $(\nabla g(u))^Tv>0$.
\end{rem}

\begin{lem}[Existence of Lagrange multipliers]\label{lagrange}
Let $u\in\R^n$, $f\in\R[\x]$, $\g\in\R[\x]^m$,
let $U$ be a neighborhood of $u$ in $\R^n$ such that $U\cap S(\g)$
is convex and not a singleton. Moreover, suppose $g_1,\ldots,g_m$ are quasiconcave at $u$, $f\ge0$ on $U\cap S(\g)$ and
\[f(u)=g_1(u)=\ldots=g_m(u)=0.\]
Then there are $\la_1,\ldots,\la_m\in\R_{\ge0}$ such that
\[\nabla f=\sum_{i=1}^m\la_i\nabla g_i(u).\]
\end{lem}

\begin{proof}
Choose $x\in U\cap S(\g)$ with $v:=x-u\ne0$: By Remark \ref{gradientpositive}, we have $(\nabla g_i(u))^Tv>0$ for $i\in\{1,\ldots,m\}$ since
$S$ is convex. Assume the required Lagrange multipliers do not exist. Then Farkas' lemma \cite[Corollary 22.3.1]{roc} yields $w\in\R^n$
such that $(\nabla f)^Tw<0$ and $((\nabla g_i)(u))^Tw\ge0$ for all $i\in\{1,\ldots,m\}$. Replacing $w$ by $w+\ep v$ for some
small $\ep>0$, we get even $((\nabla g_i)(u))^Tw>0$ for all $i\in\{1,\ldots,m\}$. Then for all sufficiently small $\de\in\R_{\ge0}$,
we have $u+\de v\in U\cap S(\g)$ but $f(u+\de v)<0$ $\lightning$.
\end{proof}

\subsection{The convex boundary}
We present the useful concept of the \emph{convex boundary} that has been introduced by Helton and Nie in \cite[Pages 762 and 775]{hn2}.

\begin{df}\label{nearconvexboundary}
Let $S\subseteq\R^n$.
We call \[\convbd S:=S\cap\partial\conv S\] the \emph{convex boundary} of $S$.
We say that $S$ \emph{has nonempty interior near its convex boundary}
if \[\convbd S\subseteq\overline{S^\circ}.\]
\end{df}

\begin{pro}\label{convbdchar}
Let $S\subseteq\R^n$. Then
\[\convbd S=\{u\in S\mid\exists\text{ linear form $f$ in }\R[\x]_1\setminus\{0\}:\forall x\in S:f(u)\le f(x)\}.\]
\end{pro}

\begin{proof}
If the affine hull of $S$ is $\R^n$, then this follows from \cite[11.6.2]{roc} and otherwise it is trivial.
\end{proof}

\subsection{Lagrange multipliers from a real closed field}

\begin{nt}
For $g\in\R[\x]$, we denote by
\[Z(g):=\{x\in\R^n\mid g(x)=0\}\]
the real zero set of $g$.
\end{nt}

\begin{lem}\label{prep1}
Let $B\subseteq\R^n$ be a closed ball in $\R^n$,
suppose that $g_1,\ldots,g_m\in\R[\x]$ are strictly quasiconcave on $B$ and set $\g:=(g_1,\ldots,g_m)$.
Then the following hold:
\begin{enumerate}[(a)]
\item $S(\g)\cap B$ is convex.
\item Every linear form from $\R[\x]\setminus\{0\}$ has at most one minimizer on $S(\g)\cap B$.
\item Let $u$ be a minimizer of the linear form $f\in\R[\x]\setminus\{0\}$ on $S(\g)\cap B$ and set
\[I:=\{i\in\{1,\ldots,m\}\mid g_i(u)=0\}.\]
Then $u$ is also minimizer of $f$ on
\[S:=\{x\in B\mid\forall i\in I:g_i(x)\ge0\}\supseteq S(\g)\cap B.\] 
\end{enumerate}
\end{lem}

\begin{proof}
(a)
Let $x,y\in S(\g)\cap B$ with $x\ne y$ and $i\in\{1,\ldots,m\}$. The polynomial \[f:=g_i(Tx+(1-T)y)\in\R[T]\]
attains a minimum $a$ on $[0,1]_\R$. We have to show $a\ge0$.
Because of $f(0)=g_i(y)\ge0$ and $f(1)=g_i(x)\ge0$, it is enough to show that this minimum is not attained
in a point $t\in(0,1)_\R$. Assume it is. Then $f'(t)=0$, i.e., $((\nabla g_i)(z))^Tv=0$ for $z:=tx+(1-t)y$
and $v:=x-y\ne0$. Since $z\in B$ and hence $g_i$ is strictly quasiconcave at $z$, it follows that
$v^T((\hess g_i)(z))v<0$, i.e., $f''(t)<0$. Then $f<a$ on a neighborhood of $t$ $\lightning$.

\medskip
(b) Suppose $x$ and $y$ are minimizers of the linear form $f\in\R[\x]\setminus\{0\}$ on $S(\g)\cap B$.
Then $x,y\in\convbd(S(\g)\cap B)$ by Proposition~\ref{convbdchar}. Since $f$ is linear, it is constant on $\conv\{x,y\}$. Hence even
\[\conv\{x,y\}\overset{\text{Prop. }\ref{convbdchar}}{\underset{\text{(a)}}\subseteq}\convbd S\overset{\text{(a)}}=S\cap\partial S
=S\cap(\overline S\setminus S^\circ)\overset{\text{$S$ closed}}=S\setminus S^\circ=\partial S.\]
Since $\conv\{x,y\}\setminus\{x,y\}\subseteq B^\circ$, we have that
$\conv\{x,y\}\setminus\{x,y\}\subseteq Z(g_1\dotsm g_m)$.
Assume now for a contradiction that $x\ne y$.
Then this implies that at least one of the $g_i$ vanishes on $\conv\{x,y\}$. Fix a corresponding $i$.
Setting $v:=y-x$, we have then $((\nabla g_i)(x))^Tv=0$ and $v^T((\hess g_i)(x))v=0$.
Since $g_i$ is strictly quasiconcave at $x$, this implies $v=0$, i.e., $x=y$ as desired.

\medskip
(c) By definition of $I$, the sets $S(\g)\cap B$ and $S$ coincide on a neighborhood of $u$ in $\R^n$. Hence $u$ is a
\emph{local} minimizer of $f$ on $S$. Since $S$ is convex by (a) and $f$ is linear, $u$ is also a (\emph{global})
minimizer of $f$ on $S$.
\end{proof}

\begin{lem}\label{lagrange2}
Suppose $B$ is a closed ball in $\R^n$,
$g_1,\ldots,g_m\in\R[\x]$ are strictly quasiconcave on $B$ and $S(\g)\cap B$
has nonempty interior. Then the following hold:
\begin{enumerate}[(a)]
\item For every real closed extension field $R$ of $\R$
and all linear forms $f\in R[\x]\setminus\{0\}$,
$f$ has a unique minimizer on $(S(\g)\cap B)_R$.
\item
For every real closed extension field $R$ of $\R$,
all linear forms $f\in R[\x]$ with \[\|\nabla f\|_2=1\]
(note that $\nabla f\in R^n$) and every
$u$ which minimizes $f$ on $(S(\g)\cap B^\circ)_R$,
there are $\la_1,\ldots,\la_m\in\O_R\cap R_{\ge0}$ with
$\la_1+\ldots+\la_m\notin\m_R$ such that \[f-f(u)-\sum_{i=1}^m\la_ig_i\in I_u^2.\]
\end{enumerate}
\end{lem}

\begin{proof}
(a) By the Tarski transfer principle, it suffices to prove the statement in the case $R=\R$.
But then the unicity follows from Lemma \ref{prep1}(b) and existence from the compactness of $S(\g)\cap B$.

\medskip
(b) Now let $R$ be a real closed field extension of $\R$, $f\in R[\x]$ a linear form with $\|\nabla f\|_2=1$ and
$u$ a minimizer of $f$ on $(S(\g)\cap B^\circ)_R$.
Set \[I:=\{i\in\{1,\ldots,m\}\mid g_i(u)=0\}\]
and define the set
\[S:=\{x\in B\mid\forall i\in I:g_i(x)\ge0\}\supseteq S(\g)\cap B\]
which is convex by  \ref{prep1}(a).
Using the Tarski transfer principle,
one shows easily that $u$ is a minimizer of $f$ on $(S\cap B^\circ)_R$ by Lemma \ref{prep1}(c).
Note also that of course $u\in\O_R^n$ and $\st(u)\in S$.

Now Lemma \ref{lagrange} says in particular that for all linear forms $\widetilde f\in\R[\x]$ and
minimizers $\widetilde u$ of $\widetilde f$ on $S\cap B^\circ$ with $\forall i\in I:g_i(\widetilde u)=0$, there is a family
$(\la_i)_{i\in I}$ in $\R_{\ge0}$ such that \[\nabla\widetilde f=\sum_{i\in I}\la_i\nabla g_i(\widetilde u).\]

Using the Tarski transfer priniciple, we see that actually for all real closed extension fields
$\widetilde R$ of $\R$, all linear forms $\widetilde f\in\widetilde R[\x]$ and
all minimizers $\widetilde u$ of $\widetilde f$ on $(S\cap B^\circ)_R$
with $\forall i\in I:g_i(\widetilde u)=0$, there is a family
$(\la_i)_{i\in I}$ in $R_{\ge0}$ such that
\[\nabla\widetilde f=\sum_{i\in I}\la_i\nabla g_i(\widetilde u).\]

We apply this to $\widetilde R:=R$, $\widetilde u:=u$, $\widetilde f:=f$ and thus obtain a family
$(\la_i)_{i\in I}$ in $R_{\ge0}$ such that
\[(*)\qquad\nabla f=\sum_{i\in I}\la_i\nabla g_i(u).\]

In order to show that $\la_i\in\O_R$ for all $i\in I$, we choose a point $x\in S^\circ\ne\emptyset$ with $\prod_{i\in I}g_i(x)\ne0$
and thus $g_i(x)>0$ for all $i\in I$. Setting $v:=x-u\in\O_R^n$, we get from $(*)$ that
$(\nabla f)^Tv=\sum_{i\in I}\la_i(\nabla g_i(u))^Tv$. Since $\st(u)\in S$ and $S$ is convex,
Remark~\ref{gradientpositive} yields $\st((\nabla g_i(u))^Tv)=(\nabla g_i(\st(u)))^T\st(v)>0$ for all $i\in I$
(use that $\st(u)\ne x$ since $g_i(\st(u))=0$ while $g_i(x)>0$).
Together with $\la_i\ge0$ for all $i\in I$, this shows $\la_i\in\O_R$ for all $i\in I$ as desired.

It now suffices to show that $\sum_{i\in I}\la_i\notin\m_R$. But this is clear since $(*)$ yields in particular
\[1=\|\nabla f\|_2\le\sum_{i\in I}\la_i\|(\nabla g_i)(u)\|_2\le\left(\sum_{i\in I}\la_i\right)\max_{i\in I}\|(\nabla g_i)(u)\|_2\]
(note that $I\ne\emptyset$ by the first inequality) which readily implies $\sum_{i\in I}\la_i\notin\m_R$.
\end{proof}

\begin{lem}\label{finitecontact}
Let $\g=(g_1,\ldots,g_m)\in\R[\x]^m$ such that $S(\g)$ is compact and has nonempty interior near its convex boundary.
Suppose that $g_i$ is strictly quasiconcave on $\convbd(S(\g))\cap Z(g_i)$ for each $i\in\{1,\dots,m\}$.
Let $R$ be real closed extension field of $\R$ and
$f\in R[\x]$ be a linear form with $\|\nabla f\|_2=1$.
Then the following hold:
\begin{enumerate}[(a)]
\item $F:=\{u\in S(\g)\mid\forall x\in S(\g):\st(f(u))\le\st(f(x))\}$
is a finite subset of $\convbd(S(\g))$.
\item We have $S:=(S(\g))_R\subseteq\O_R^n$ and
$f$ has a unique minimizer $x_u$ on \[\{x\in S\mid \st(x)=u\}\] for each $u\in F$.
\item For every $u\in F$, there are
$\la_{u1},\ldots,\la_{um}\in\O_R\cap R_{\ge0}$ with
$\la_{u1}+\ldots+\la_{um}\notin\m_R$ such that \[f-f(x_u)-\sum_{i=1}^m\la_{ui}g_i\in I_{x_u}^2.\]
\end{enumerate}
\end{lem}

\begin{proof}
(a) Obviously $\st(f)\ne0$ and hence \[F=\{u\in S(\g)\mid\forall x\in S(\g):(\st(f))(u)\le(\st(f))(x)\}\subseteq
\convbd(S(\g))\] by Proposition \ref{convbdchar}.
We now prove that $F$ is finite.
WLOG $S(\g)\ne\emptyset$. Set \[a:=\min\{(\st(f))(x)\mid x\in S(\g)\}\] so that
\[F=\{u\in S(\g)\mid(\st(f))(u)=a\}.\] By compactness of $S(\g)$, it is enough to show that
every $x\in S(\g)$ possesses a neighborhood $U$ in $S(\g)$ such that $U\cap F\subseteq\{x\}$.
This is trivial for the points in $S(\g)\setminus F$. So consider an arbitrary point $x\in F$.
Since $x\in\convbd(S(\g))$, each $g_i$ is positive or strictly
quasiconcave at $x$. According to Lemma \ref{strictly2neighbor},
we can choose a closed ball $B$ of positive radius around
$x$ in $\R^n$ such that each $g_i$ is positive or strictly
quasiconcave even on $B$. By Lemma \ref{prep1}(b), $\st(f)$ has at most one minimizer on $U:=S(\g)\cap B$,
namely $x$, i.e., $U\cap F\subseteq\{x\}$.

\medskip
(b) First observe that
$S:=(S(\g))_R\subseteq\O_R^n$ since the
transfer from $\R$ to $R$ is an isomorphism of Boolean algebras:
Choosing $N\in\N$ with $S(\g)\subseteq[-N,N]^n$, we have $S\subseteq[-N,N]_R^n\subseteq\O_R^n$.

Now we fix $u\in F$ and we show that $f$ has a unique minimizer on
\[A:=\{x\in S\mid\st(x)=u\}.\]
Choose $\ep\in\R_{>0}$ such that
each $g_i$ is strictly quasiconcave or positive on
the ball \[B:=B_\ep(u).\]
Since $u\in\convbd(S(\g))\subseteq\overline{S(\g)^\circ}$,
Lemma \ref{lagrange2}(a) says that $f$ has a unique minimizer $x$ on $(S(\g)\cap B)_R$.
Because of $A\subseteq(S(\g)\cap B)_R$, it is thus enough to show $x\in A$.
Note that $u\in F\cap B\subseteq S(\g)\cap B\subseteq(S(\g)\cap B)_R$ and thus
$f(x)\le f(u)$. This implies $\st(f(\st(x)))=\st(f(x))\le\st(f(u))$ which yields together with $\st(x)\in S(\g)$ that
$\st(x)\in F$ (and $\st(f(\st(x)))=\st(f(u))$). Again by Lemma \ref{lagrange2}(a), $\st(f)$ has a unique minimizer on
$S(\g)\cap B$ . But $u$ and $\st(x)$ are both a minimizer of $\st(f)$ on $S(\g)\cap B$ (note that $\st(x)\in S(\g)\cap B$).
Hence $u=\st(x)$ and thus $x\in A$ as desired.

\medskip
(c) Fix $u\in F$. Choose again $\ep\in\R_{>0}$ such that
each $g_i$ is strictly quasiconcave or positive on
the ball $B:=B_\ep(u)$ and such that $B\cap F=\{u\}$.
Since $x_u$ obviously minimizes $f$ on $(S(\g)\cap B)_R$, we get the necessary
Lagrange multipliers by Lemma~\ref{lagrange2}(b).
\end{proof}

\subsection{Linear polynomials and quadratic modules}

\begin{rem}\label{calculusexo}
For all $k\in\N$ and $x\in[0,1]_\R$, we have $x(1-x)^k\le\frac1k$.
\end{rem}

The main geometric idea in the proof of the following theorem is as follows: Consider a hyperplane that isolates a basic closed
semialgebraic subset of $\R^n$ and that is defined over a real closed extension field of $\R$. Because we want to apply
Theorem $\ref{mainrep}$ to get a sums of squares ``isolation certificate'', the points where the hyperplane gets infinitesimally
close to the set pose problems unless the hyperplane \emph{exactly} touches the set in the respective point. The idea is
to find a nonlinear infinitesimal deformation of the hyperplane so that all ``infinitesimally near points'' becoming ``touching points''.
This would be easy if there is at most one ``infinitesimally near point'' but since we deal in this article
with \emph{not necessarily convex} basic closed semialgebraic sets, it is crucial to cope with several such points.

\begin{thm}\label{linearstability}
Let $\g\in\R[\x]^m$ such that $M(\g)$ is Archimedean and suppose that $S(\g)$ has nonempty interior near its convex boundary.
Suppose that $g_i$ is strictly quasiconcave on $(\convbd S)\cap Z(g_i)$ for each
$i\in\{1,\dots,m\}$. Let $R$ be a real closed extension field of $\R$ and $\ell\in\O_R[\x]_1$ such that $\ell\ge0$ on $S_R$.
Then $\ell$ lies in the quadratic module generated by $\g$ in $\O_R[\x]$.
\end{thm}

\begin{proof}
We will apply Theorem $\ref{mainrep}$. Since $S(\g)$ is compact, we can rescale the $g_i$ and suppose WLOG that
\[g_i\le1\text{ on }S(\g)\]
for $i\in\{1,\ldots,m\}$.
Let $M$ denote the quadratic module generated by $\g$ in $\O_R[\x]$.
Since $M(\g)$ is Archimedean, also $M$ is Archimedean by 
Propositions \ref{archmodulechar}(b) and \ref{archmodulecharrcf}(b). Moreover, $S$ could now alternatively be defined
from $M$ as in Theorem~$\ref{mainrep}$. Write
\[\ell=f-c\]
with a linear form $f\in\O_R[\x]$ and $c\in\O_R$.
By a rescaling argument, we can suppose that at least one of the coefficients of $\ell$ lies
in $\O_R^\times$.
If $\st(\ell(x))>0$ for all $x\in S(\g)$, then Theorem $\ref{mainrep}$ applied to $\ell$ with $k=0$ yields $\ell\in M$
and we are done. Hence we can from now on suppose that there is some $u\in S$ with $\st(\ell(u))=0$. For such an $u$,
we have $\st(c)=\st(f(u))$ so that at least one coefficient of $f$ must lie in $\O_R^\times$.
By another rescaling, we now can suppose WLOG that $\|\nabla f\|_2=1$. Now we are in the situation of
Lemma~\ref{finitecontact} and we define
\[F,\quad (x_u)_{u\in F}\quad\text{and}\quad
(\la_{ui})_{(u,i)\in F\times\{1,\ldots,m\}}\]
accordingly. Note that
\[F=\{u\in S(\g)\mid\st(\ell(u))=0\}\ne\emptyset\]
since $\st(\ell(x))\ge0$ for all $x\in S(\g)$.
We have $f(x_u)-c=\ell(x_u)\ge0$ and \[\st(f(x_u)-c)=\st(\ell(u))=0\] for all $u\in F$.
Hence $f(x_u)-c\in\m_R\cap R_{\ge0}$ for all $u\in F$. We thus have
\[\ell-\underbrace{(f(x_u)-c)}_{=:\la_{u0}\in\m_R\cap R_{\ge0}}-\sum_{i=1}^m\underbrace{\la_{ui}}_{\in\O_R\cap R_{\ge0}}g_i\in I_{x_u}^2\]
for all $u\in F$ by the Lemmata \ref{finitecontact}(c) and \ref{membershipix}. Evaluating this in $x_u$ (and using $g_i(x_u)\ge0$) yields
\begin{gather}
\tag{$*$}g_i(x_u)\ne0\implies\la_{ui}=0\qquad\text{and thus}\\
\tag{$**$}\la_{ui}g_i\equiv_{I_{x_u}^2}\la_{ui}g_i(1-g_i)^k
\end{gather}
for all $u\in F$, $i\in\{1,\ldots,m\}$ and $k\in\N$.
By the Chinese remainder theorem, we find polynomials $s_0,\ldots,s_m\in\O_R[\x]$
such that $s_i\equiv_{I_{x_u}^3}\sqrt{\la_{ui}}\in\O_R$ for all $u\in F$ and $i\in\{0,\ldots,m\}$
because the ideals $I_{x_u}^3$ ($u\in F$) are pairwise coprime by Lemma~\ref{coprime} (use that
$\st(x_u)=u\ne v=\st(x_v)$ for all $u,v\in F$ with $u\ne v$).
By an easy scaling argument, we can even guarantee that the coefficients
of $s_0$ lie in $\m_R$ since $\sqrt\la_{u0}\in\m_R$. Then we have
\begin{equation}
\tag{$***$}s_i^2\equiv_{I_{x_u}^3}\la_{ui}
\end{equation}
which means in other words
\[s_i^2(x_u)=\la_{ui},\qquad(\nabla(s_i^2))(x_u)=0\qquad\text{and}\qquad(\hess(s_i^2))(x_u)=0\]
for all $i\in\{0,\ldots,m\}$  and $k\in\N$.
It suffices to show that there is $k\in\N$ such that the polynomial
\[\ell-s_0^2-\sum_{i=1}^ms_i^2(1-g_i)^{2k}g_i\overset{(***)}{\underset{(**)}\in}\bigcap_{u\in F}I_{x_u}^2\]
lies in $M$ since this implies immediately $\ell\in M$. By Theorem $\ref{mainrep}$, this task reduces to find $k\in\N$ such that
$f_k>0$ on $S(\g)\setminus F$ and $(\hess(f_k))(u)\succ0$ for all $u\in F$ where
\[f_k:=\st(\ell)-\sum_{i=1}^m\st(s_i^2)(1-g_i)^{2k}g_i\in\R[\x]
\]
is the standard part of this polynomial. Note for later use that $f_k$ and $\nabla f_k$ vanish on $F$ for all $k\in\N$.
In order to find such a $k$, we calculate
\begin{align*}
(\hess f_k)(u)&\overset{(***)}=-\sum_{i=1}^m\st(\la_{ui})\hess((1-g_i)^{2k}g_i)(u)\\
&\ \ \overset{\text{Rem. }\ref{derexo}}{\underset{(*)}=}\sum_{i=1}^m\st(\la_{ui})(4k(\nabla g_i)(\nabla g_i)^T-\hess g_i)(u)
\end{align*}
for $u\in F$ and $k\in\N$.
By Lemma \ref{quasi2concave} we can choose $k\in\N$ such that
$g_i(1-g_i)^{2k}$ is strictly concave on $\{x\in F\mid g_i(x)=0\}$ for $i\in\{1,\ldots,m\}$. Since \[\st(\la_1)+\ldots+\st(\la_m)>0\]
by Lemma \ref{finitecontact}(c), we get together with $(*)$ and Remark~\ref{derexo} that for all sufficiently large $k$, we
have $(\hess f_k)(u)\succ0$ for all $u\in F$. In particular, we can choose $k_0\in\N$ such that
$\hess(f_{k_0})(u)\succ0$ for all $u\in F$. Since $f_{k_0}$ and $\nabla f_{k_0}$ vanish on $F$,
we have by elementary analysis that there is an open subset $U$ of $\R^n$ containing $F$ such that $f_{k_0}\ge0$ on $U$.
Then $S(\g)\setminus U$ is compact so that we can choose $N\in\N$ with $\st(\ell)\ge\frac1N$ and
$\st(s_i^2)\le N$ on $S(\g)\setminus U$. Then $f_k\ge\frac1N-m\frac N{2k}$ on $S(\g)\setminus U$ by
Remark \ref{calculusexo} since $0\le g_i\le1$ on $S$ for all $i\in\{1,\ldots,m\}$.
For all sufficiently large $k\in\N$ with $k\ge k_0$, we now have $f_k>0$ on $S\setminus U$ and
because of $f_k\ge f_{k_0}>0$ on $S(\g)\cap U$ (use again that $0\le g_i\le1$ on $S(\g)$) even $f_k>0$ on $S(\g)$.
\end{proof}

For the applications in this article
we do no longer need to work over the ring $\O_R$ of finite elements of a real closed extension field
$R$ of $\R$ (even though this could have other potential applications in the future, cf. Remark \ref{commentonbounds}(b) above)
but can directly work over $\R$. That is why we draw the following corollary.

\begin{cor}\label{linearstabilitycor}
Let $\g\in\R[\x]^m$ such that $M(\g)$ is Archimedean and suppose that $S(\g)$ has nonempty interior near its convex boundary.
Suppose that $g_i$ is strictly quasiconcave on $(\convbd S)\cap Z(g_i)$ for each
$i\in\{1,\dots,m\}$. Let $R$ be a real closed extension field of $\R$ and $\ell\in R[\x]_1$ such that $\ell\ge0$ on $S_R$.
Then $\ell$ lies in the quadratic module generated by $\g$ in $R[\x]$.
\end{cor}

\subsection{Linear polynomials and truncated quadratic modules}

Going now from real closed fields back to the reals,
the general tools from Subsection \ref{subs:tarski}, enable us now
to turn Corollary \ref{linearstabilitycor} into Corollary \ref{linearstabilitybound}
just like we have turned the statement Corollary \ref{mainrep3}
into Theorem \ref{putinarzerosdegreebound} earlier on.

\begin{cor}\label{linearstabilitybound}
Let $\g\in\R[\x]^m$ such that $M(\g)$ is Archimedean and suppose that $S(\g)$ has nonempty interior near its convex boundary.
Suppose that $g_i$ is strictly quasiconcave on $(\convbd S)\cap Z(g_i)$ for each
$i\in\{1,\dots,m\}$.
Then there exists \[d\in\N\] such that for all $\ell\in\R[\x]_1$ with $\ell\ge0$ on $S$, we have \[\ell\in M_d(\g).\]
\end{cor}

\begin{proof}
For each $d\in\N$, consider the set $S_d\subseteq\R^{n+1}$ of all
$a=(a_0,a_1,\ldots,a_n)\in\R^{n+1}$
such that whenever \[\forall x\in S:a_1x_1+\ldots+a_nx_n+a_0\ge0\]
holds, the polynomial $a_1X_1+\ldots+a_nX_n+a_0$ is a sum
of $d$ elements from $\R[\x]$ where each term in the sum is of degree at most $d$
and is of the form $p^2g_i$ with $p\in\R[\x]$ and $i\in\{0,\dots,m\}$ where $g_0:=1\in\R[\x]$
(cf. Remark \ref{commentonbounds}(a)).
By real quantifier elimination, it is easy to see that each $S_d$ is a semialgebraic set,
see Subsection \ref{subs:tarski} above. Obviously, $S_1\subseteq S_2\subseteq S_3
\subseteq\ldots$.
We want to show that $S_d=\R^{n+1}$ for some
$d\in\N$. By the finiteness theorem (see Subsection \ref{subs:tarski} above), it suffices to show 
$\bigcup_{d\in\N}(S_d)_R=R^{n+1}$ for all real closed extension fields $R$ of $\R$. But this follows from
Corollary \ref{linearstabilitycor}.
\end{proof}

\subsection{Consequences for the Lasserre relaxation of a system of polynomial inequalities}\label{subs:cons-solving}

Finally, we harvest our second crop and present our
application to solving systems of polynomial inequalities
that we have announced and discussed already in Subsection \ref{subs:towardssolving}.

\begin{cor}\label{lasserreexact}
Let $\g\in\R[\x]^m$ such that $M(\g)$ is Archimedean and suppose that $S(\g)$ has nonempty interior near its convex boundary.
Suppose that $g_i$ is strictly quasiconcave on $(\convbd S)\cap Z(g_i)$ for each $i\in\{1,\dots,m\}$.
Then there exists $d\in\N$ such that \[S(\g)=\conv S_d(\g).\]
\end{cor}

\begin{rem}\label{mildnatural}
By Remark \ref{whatmeansquasiconcavity} (see \cite[Proposition 3.4]{ks} or \cite{hn1,hn2} for more details), it becomes clear that
the hypothesis made in Theorem \ref{linearstability} and in its Corollaries \ref{linearstabilitycor}, \ref{linearstabilitybound} and
\ref{lasserreexact} on the $g_i$ to be strictly quasiconcave on \[(\convbd S)\cap Z(g_i)\] is mild and natural.
On the other hand, it even excludes linear defining polynomials but this is because the statement becomes in general false
for nonconvex $S(\g)$ in this case as \cite[Example 4.10]{ks} shows.
\end{rem}

\begin{ex}\label{onedimex}
The requirement that $S(\g)$ has nonempty interior near its convex boundary cannot be dropped. Consider $n=1$ and
$g=X(1-X)(X-2)^2\in\R[\x]$. Then $S(g)=[0,1]\cup\{2\}$ is compact and moreover $M(g)$ is Archimedean, for example by
Remark \ref{prestel}. Also $g$ is strictly quasiconcave
on $\convbd S(g)=\{0,2\}$ since $g'(0)=4\ne0$, $g'(2)=0$ and $g''(2)=-4<0$.

It is an immediate consequence of a result of Gouveia and Netzer
\cite[Proposition 4.1]{gn} (see also \cite[Theorem 4.9]{ks}) that $g$ has no exact Lasserre relaxation.
By inspection of the proof of Gouveia and Netzer, one sees that each Lasserre relaxation contains a right neighborhood of $2$.
\end{ex}

In the situation of Corollary \ref{lasserreexact}, Helton and Nie proved that $S(\g)$ is a projection of a spectrahedron
\cite[Theorem 3.3]{hn2}. To obtain this result, they showed that certain Lasserre relaxations for
``small boundary pieces'' of $S(\g)$ become exact and they glue together these pieces
obtained non-constructively by a compactness argument. In particular, their arguments do not give any algorithm of how
to compute such a semidefinite lifting whereas we show with our new technique that the initial Lasserre relaxation itself just
succeeds. With again completely different techniques, Helton and Nie tried to restrict themselves to a single Lasserre relaxation
by altering the description of $S(\g)$ but could only achieve
very technical partial results going in the direction of Corollary \ref{lasserreexact}
above \cite[Section 5]{hn2}. This is discussed in more detail at the end of the introduction of \cite{ks}.

Although already stated in Subsection
\ref{subs:towardssolving}, we remind the reader again that our Corollary \ref{lasserreexact} is weaker than
\cite[Main Theorem 4.8]{ks} in the case of a \emph{convex} basic closed semialgebraic set.
We do not know how to extend our method using real closed fields to prove \cite[Main Theorem 4.8]{ks}
and, conversely, it seems that the techniques in \cite{ks} that are clever refinements of the ideas of Helton and Nie \cite{hn1,hn2}
do not seem to work here at all if $S(\g)$ is nonconvex.

A major recent result of Scheiderer that complements our positive results on Lasserre relaxations on the negative side is that
\emph{not} every convex semialgebraic set is a projection of a spectrahedron \cite{s2}.

\subsection{What about quadratic polynomials?}
Finally, we show that Theorem \ref{linearstability} and its corollaries do not extend from linear to quadratic polynomials. We need the following preparation: 

\begin{pro}\label{como}
Let $\g=(g_1,\ldots,g_m)\in\R[\x]^m$ and $x\in S(\g)$. Suppose that there exists $y\in S(\g)\setminus\{x\}$ such that $\conv\{x,y\}\subseteq S(\g)$.
Set \[I:=\{i\in\{1,\ldots,m\}\mid g_i(x)=0\}\] and suppose that $g_i$ is strictly quasiconcave at $x$ for all $i\in I$.
Consider (taking $R:=\R$ in Definition~\ref{adamshom}) \[I_x=(X_1-x_1,\ldots,X_n-x_n)\subseteq\R[\x].\]
Fix $v\in\R^n$ and define $\ph\colon\R[\x]\to\R, p\mapsto v^T(\hess p)(x)v$.
Then \[\ph\left(M(\g)\cap I_x^2\right)\subseteq\R_{\ge0}.\]
\end{pro}

\begin{proof}
Set $g_0:=1\in\R[\x]$ and consider $s_0,\ldots,s_m\in\sum\R[\x]^2$ such that
\[
p:=\sum_{i=0}^m s_ig_i \in I_x^2.
\]
We have to show $\ph(p)\ge0$. We will show this even under the weaker condition that $s_i\ge0$ on $\R^n$ instead of
$s_i\in\sum\R[\x]^2$ for $i\in\{0,\dots,m\}$.
We have \[\sum_{\substack{i=0\\i\notin I}}^m\underbrace{s_i(x)}_{\ge0}\underbrace{g_i(x)}_{>0}=0\]
and therefore $s_i\in I_x^2$ for each $i\in\{0,\dots,m\}\setminus I$. In particular,
\[q:=\sum_{i\in I}s_ig_i\in I_x^2\]
and
\[\ph\left(s_ig_i\right)=\underbrace{g_i(x)}_{>0}v^T\underbrace{(\hess s_i)(x)}_{\succeq0}v\ge0
\]
for all $i\in\{0,\dots,m\}\setminus I$. Hence $\ph(p)\ge\ph(q)$.
We will show that $\ph(q)=0$. In fact, we will even show that $s_ig_i\in I_x^3$ for all $i\in I$.
It suffices to show that $s_i\in I_x^2$ for all $i\in I$.
Setting $v:=y-x$, we have $(\nabla g_i(x))^Tv>0$ for $i\in I$ by Remark \ref{gradientpositive}. Hence
\[\sum_{i\in I}\underbrace{s_i(x)}_{\ge0}\underbrace{(\nabla g_i(x))^Tv}_{>0}=(\nabla q(x))^T=0.\]
This implies $s_i(x)=0$ and thus $s_i\in I_x^2$ and hence $s_ig_i\in I_x^3$ for $i\in I$, using that $s_i\ge0$ on $\R^n$.
\end{proof}

Now we come to the promised example.

\begin{ex}
Consider
\begin{align*}
g&:=1-(X-1)^2-Y^2\in\R[X,Y]\qquad\text{and}\\
h&:=1-X^2-(Y-1)^2\in\R[X,Y]
\end{align*}
One shows easily that $M(g,h)$ is Archimedean (for example by Remark \ref{prestel}). We have $X\ge0$ on the ball $S(g)$ and
$Y\ge0$ on the ball $S(h)$ and thus $XY\ge0$ on $S(g,h)$. The Hessian of both $g$ and $h$ is constantly $-2I_2\prec0$. Hence
$g$ and $h$ are strictly concave on $\R^2$. Set $v:=\left(\begin{smallmatrix}1\\-1\end{smallmatrix}\right)\in\R^2$ and consider
\[\ph\colon\R[\x]\to\R, p\mapsto v^T(\hess p)(x)v.\]
By Proposition \ref{como}, we have $\ph(M(g,h)\cap I_x^2)\subseteq\R_{\ge0}$. On the other hand
$\ph(XY)=-2<0$. Since $XY\in I_0^2$, this shows that $XY\notin M(g,h)$.
\end{ex}

\section*{Acknowledgments}
\noindent
Both authors were supported by the DFG grant SCHW 1723/1-1.

\end{document}